\documentclass{article}[11pt]
\usepackage[utf8]{inputenc}
\usepackage{amsmath}
\usepackage{amsthm}
\usepackage{amssymb}
\usepackage{comment}
\usepackage{bbm}
\usepackage{float}
\usepackage{geometry}
\usepackage{graphicx}
\usepackage{subcaption}
\usepackage[usenames,dvipsnames]{xcolor}
\usepackage{dsfont,url}
\usepackage{tikz}
\usepackage[normalem]{ulem}

\usetikzlibrary{positioning}
\usetikzlibrary{arrows}
\usetikzlibrary{arrows.meta}
\usetikzlibrary{calc}

\usepackage{mathtools}
\mathtoolsset{showonlyrefs}
\usepackage{hyperref}

\newcommand{\R}{\mathbb{R}}
\newcommand{\RO}{\mathcal{R}_0}
\newcommand{\f}{\varphi}

\newcommand{\propref}[1]{Proposition~\ref{proposition:#1}}
\newcommand{\proplab}[1]{\label{proposition:#1}}

\newtheorem{theorem}{Theorem}

\newtheorem{proposition}[theorem]{Proposition}

\newtheorem{lemma}[theorem]{Lemma}

\newtheorem{remark}{Remark}

\newcommand{\andrea}[1]{\textcolor{blue}{#1}}

\newcommand{\panos}[1]{\textcolor{magenta}{#1}}

\newcommand{\intr}{^{\wedge}}

\title{A geometric analysis of the SIRS model with secondary infections}
\author{Panagiotis Kaklamanos{$^1$}, Andrea Pugliese{$^2$}, Mattia Sensi{$^{3,4,*}$}, Sara Sottile{$^{2}$}\\[1em]
\small $^1$Maxwell Institute for Mathematical Sciences and School of Mathematics,
University of Edinburgh,\\\small James Clerk Maxwell Building, King's Buildings, Peter Guthrie Tait Road,\\\small 
Edinburgh, EH9 3FD,
United Kingdom\\
\small $^2$Dipartimento di Matematica, Universit\`a degli Studi di Trento\\
\small
Via Sommarive 14, 38123 Povo (Trento), Italy,\\
\small $^3$MathNeuro Team, Inria at Universit\' e C\^ote d'Azur, 2004 Rte des Lucioles, 06410 Biot, France\\
\small $^4$Politecnico di Torino, Corso Duca degli Abruzzi 24, 10129 Torino Italy\\
\small $^*$Corresponding author. Email: {\texttt{mattia.sensi@polito.it}}}
\date{\today}
\begin{document}

\maketitle

\begin{abstract}
    {We propose a compartmental model for a disease with temporary immunity and secondary infections. From our assumptions on the parameters involved in the model, the system naturally evolves in three time scales. We characterize the equilibria of the system and analyze their stability. We find conditions for the existence of two endemic equilibria, for some cases in which $\mathcal{R}_0 < 1$.
    Then, we unravel the interplay of the three time scales, providing conditions to foresee whether the system evolves in all three scales, or only in the fast and the intermediate ones. We conclude with numerical simulations and bifurcation analysis, to complement our analytical results.}
\end{abstract}


\section{Introduction}

The foundation of the mathematical epidemics modelling based on compartmental models goes back to the XX century, based on the pioneering work by Kermack and McKendrick \cite{kermack1927contribution}. Since then, several generalizations have been proposed as attempts to develop more realistic models that take several other factors into account.

A fundamental distinction in epidemic model is between SIR and SIS models \cite{three_epi_mod}, with the former modelling infections providing complete immunity, and the latter those not providing any immunity. In reality, immunity may be only partial, making infections less likely but not impossible, or protecting from some consequences of infection (disease) but not from infection itself. Partial immunity may be due to immunity waning with time since infection, to a secondary infection caused by a pathogen similar but not identical to that of the primary infection, or simply to the limited immunity induced by the primary infection. 

Partial immunity and reinfections have drawn strong interest during the COVID-19 pandemic, but its causes are being debated \cite{Danchin2022,Gousseff20}, and the pattern is certainly very complex \cite{science_covid_immunity}. Several mathematical models have been devoted to partial immunity caused by infection waning, with the additional possibility of immunity boosting \cite{Dafilis2012a,jardon2021geometric1,Lavine2011a,Opoku-Sarkodie2022}, or to partial cross-immunity to a heterologous strain \cite{AndreasenLinLevin97,Nuno2005,Laurie2015,Castillo-Chavez:1989}. 

Few epidemic models have instead been devoted to the case where more that one infection episode is needed to provide complete immunity, although this is a mechanism recognized in the immunological literature \cite{Pascucci2021,Zarnitsyna2016} and is indeed consistent with the practice of performing vaccination in two doses. One may refer to \cite{Lopman14} in which the authors analyse data on \textit{Norovirus} prevalence, assuming that individuals can be infected any number of times, but only the first infection is symptomatic; or to \cite{Le2021} in which the impact of different assumptions about infection-derived immunity on disease dynamics is assessed.

In this paper, we consider an SIR epidemic model with secondary infections; after a primary infection, individuals have a strong transient immunity, at the end of which they become partially immune (i.e., partially susceptible) and may contract the disease again. A secondary infection provides a complete immunity, which however decays with time to  partial immunity; this too decays with time to complete susceptibility.  The model is described in detail in Sec.~\ref{sec_model}. Here it suffices to say that the model involves three different time scales: a fast time-scale (of the order of days) for the infections, an intermediate time-scale (of the order of months) for the transient immunity after a primary infection, and a slow  time-scale (of the order of years) in which complete or partial immunity are lost. Two important parameters determining the epidemic dynamics are $\nu$, the relative susceptibility of partially immune individuals, and $\alpha$ the relative infectiousness of secondary infections. If $\nu = 0$, secondary infections are impossible and the model reduces to an SIRS model (with gamma-distributed immune period); if $\alpha = 0$, secondary infections do not contribute to the force of infection, and the model reduces to an SIRWS with immunity waning and boosting, except that after a primary infection, individuals are only weakly immune.

In a recent paper \cite{steindorf2022modeling}, the authors consider a model allowing for secondary infections, with assumptions very similar to ours. The differences in the assumptions are that in \cite{steindorf2022modeling} $\nu$ is equal to $1$ (no difference in susceptibility between susceptible and partial immune individuals) and immunity does not wane; on the other hand, the authors consider host  demography (which we neglect for the sake of simplicity). Especially, the main focus of \cite{steindorf2022modeling} is the numerical exploration of model solutions, and numerical bifurcation analysis. The focus of the present paper is instead on exploiting the differences in time-scales to gather an analytical understanding of the model dynamics.

It has to be noted that primary and secondary infections are usually considered in models with multiple strains \cite{aguiar2008epidemiology,aguiar2011role,kooi2013bifurcation}, which, under conditions of symmetry reduce to models very similar to the one we consider here.

The presence of very different time-scales is typical of epidemic models. Consider, for instance, models which include both disease and demographic dynamics: typically, infectious periods have a much shorter duration than the average lifespan of the individuals in the population (weeks vs. years) \cite{andreasen1993effect,jardon2021geometric1,jardon2021geometric2}. Individuals behaviour or mobility may also evolve much faster than epidemics; several papers focus on this, both in continuous  \cite{castillo2016perspectives,dellamarca2023geometric,schecter2021geometric} and discrete time \cite{de2020discrete,bravo2021discrete}. As a further example, in vector-borne diseases the time scale associated with the vector dynamics is typically faster than host dynamics;  this difference is taken into account in some recent papers \cite{aguiar2021time,rashkov2021complexity,rashkov2019role} in which the authors perform the analysis using both the Quasi-Steady-State Approximation (QSSA) and Geometric Singular Perturbation Theory (GSPT).

The existence of different time scales is exploited, in a context  somewhat similar to the present paper, in \cite{rashkov2021complexity}; there a two-strain host-vector model is considered, leading (under some simplifying assumptions, such as the irrelevance of the order of infections) to a very high dimensional system (11 equations);  a dimensionality reduction  is then obtained through a quasi-steady state approximation, exploiting a natural difference in time scales between host and vector dynamics.

In this paper, we focus on the interplay of the three time scales involved in the system, using techniques from GSPT. A thorough description of the techniques we use can be found in \cite{jones1995geometric} or \cite{kuehn2015multiple}; for a concise introduction, we refer to the introductory sections of \cite{jardon2021geometric1}. 

The paper is organised as follows. In Section \ref{sec_model}, we introduce and describe a compartmental model for SIRS diseases with secondary infections. In Section \ref{sec_eqstab}, we first study the (local and global) stability of the Disease-Free Equilibrium (DFE) in terms  of the Basic Reproduction Number $\RO$, appropriately defined. Then, we  discuss the existence of endemic equilibria of the system, finding the conditions under which the system admits a unique positive equilibrium or two. In Sections \ref{sec_fast}, \ref{sec_inter}, and \ref{sec_slow}, we study the fast, intermediate, and slow dynamics of the model, respectively, in the context of GSPT. In particular, in Section \ref{sec:entrex} we introduce the entry-exit function and we give conditions for which the system enters the slow time scale or re-enters the fast scale from the intermediate one.  In Section \ref{sec_map} we define two discrete maps which summarize the behaviour of the system. The first describes the fast scale, the second describes only the intermediate or the intermediate and the slow scales, depending on the cases. Section \ref{sec:numeric} is devoted to numerical explorations. In particular, in Section \ref{sec_bifurc} we carry on the bifurcation analysis on the system and in Section \ref{sec_num} we perform numerical simulations in the case in which the systems admits both the Disease-Free Equilibrium and two endemic equilibria, to demonstrate the behaviour of the convergence of the trajectories.
We conclude the paper with a discussion in Section \ref{sec_conc}.  {Check these final sentences once we're sure about what we include}

\section{The model}\label{sec_model}
In this section, we propose a novel compartmental model for SIRS infections with secondary infections.
We partition the total population in six compartments, with respect to an ongoing epidemic:
\begin{itemize}
    \item $S$ represents the totally susceptible individuals;
    \item $I$ represents  individuals with a primary infection;
    \item $T$ represents the temporarily immune individuals, who recently recovered from a primary infection;
    \item $P$ represents the partially susceptible individuals, who have already recovered from a primary infection and lost the transient immunity;
    \item $Y$ represents individuals with a secondary infection;
    \item $R$ represents individuals who have recovered from a second infection, and are completely immune.
\end{itemize}
For the sake of simplicity, we do not consider demography in our model. We denote with $N=S+I+T+P+Y+R$ the total population. The system of ODEs, before further simplifications, is the following:
\begin{align}\label{sys}
    S'(t) =&\; -\beta \frac{S(t)}{N(t)}(I(t)+\alpha Y(t))+\eta_1 P(t), \nonumber \\
    I'(t) =&\; \beta \frac{S(t)}{N(t)}(I(t)+\alpha Y(t)) -\gamma_1 I(t),  \nonumber\\
    T'(t) =&\; \gamma_1 I(t) -\varepsilon T(t),\\
    P'(t) =&\; \varepsilon T(t) -\nu \beta \frac{P(t)}{N(t)}(I(t)+\alpha Y(t))-\eta_1 P(t) + \eta_2 R(t), \nonumber\\
    Y'(t) =&\; \nu \beta \frac{P(t)}{N(t)}(I(t)+\alpha Y(t)) -\gamma_2 Y(t), \nonumber\\
    R'(t) =&\; \gamma_2 Y(t) -\eta_2 R(t),  \nonumber
\end{align}
where the $'$ indicates the derivative with respect to the fast time scale $t$. The parameters of the system are the following: 
\begin{itemize}
    \item $\beta$ is the rate at which totally susceptible are infected by individuals in a primary infection;
    \item $\alpha$ is the relative infectiousness of individuals in a secondary infection, compared to those in a primary infection;
    \item $\nu$ is the relative susceptibility of partially immune individuals, compared to susceptibles;
   \item $\gamma_1$ is the recovery rate from primary infections, meaning on average a primary infection lasts $1/\gamma_1$;
     \item $\gamma_2$ is the recovery rate from secondary infections, which on average last $1/\gamma_2$;
    \item $0<\varepsilon \ll 1$ is the loss rate of temporary immunity;
    \item $0<\eta_1\ll \varepsilon$ is the loss rate of partial protection  ($P\rightarrow S$);
    \item $0<\eta_2\ll \varepsilon$ is the loss rate of total protection  ($R\rightarrow P$).
\end{itemize}
All the parameters are assumed to be positive. Since the  total population remains constant, as can be seen by observing $N'(t)=0$, we can divide all variables by $N$, which is equivalent to assuming
$N=1$.  

To present the three time scales involved more clearly, we substitute $\eta_1=\eta_2=\delta \varepsilon$, with $0<\delta,\varepsilon \ll 1$, having assumed, for the sake of simplicity, $\eta_1=\eta_2$.   The system is now in non-standard GSPT form with three-time scales, with $\varepsilon$ and $\delta$ representing our perturbation parameters, and hence the ratios between the time scales involved:
\begin{align}\label{sys_simpler}
    S'(t) =&\; -\beta S(t)(I(t)+\alpha Y(t))+\delta \varepsilon P(t), \nonumber \\
    I'(t) =&\; \beta S(t)(I(t)+\alpha Y(t)) -\gamma_1 I(t),  \nonumber\\
    T'(t) =&\; \gamma_1 I(t) -\varepsilon T(t),\\
    P'(t) =&\; \varepsilon T(t) -\nu \beta P(t)(I(t)+\alpha Y(t))-\delta \varepsilon P(t) + \delta \varepsilon R(t), \nonumber\\
    Y'(t) =&\; \nu \beta P(t)(I(t)+\alpha Y(t)) -\gamma_2 Y(t), \nonumber\\
    R'(t) =&\; \gamma_2 Y(t) -\delta \varepsilon R(t).  \nonumber
\end{align}
System \eqref{sys_simpler} evolves in the biologically relevant region
\begin{equation}\label{region}
    \Tilde{\Delta} := \{  (S,I,T,P,Y,R)\in \mathbb{R}^6\, | \, S,I,T,P,Y,R\geq 0,\   S+I+T+P+Y+R= 1\}.
\end{equation}

In the following, we drop the dependence of the compartments $S$, $I$, $T$, $P$, $Y$ and $R$ on the time variables, for ease of notation. We specify whenever the time variable is changed as a consequence of time rescaling.

\begin{center}
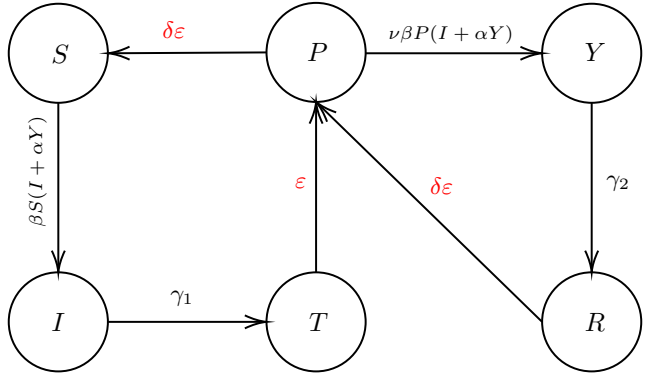
\begin{figure}[ht!]
\centering
\tikzset{every picture/.style={line width=0.75pt}} 
\begin{tikzpicture}[x=0.75pt,y=0.75pt,yscale=-1,xscale=1]
\draw   (170,36) .. controls (170,22.19) and (181.19,11) .. (195,11) .. controls (208.81,11) and (220,22.19) .. (220,36) .. controls (220,49.81) and (208.81,61) .. (195,61) .. controls (181.19,61) and (170,49.81) .. (170,36) -- cycle ;
\draw   (170,171.57) .. controls (170,157.76) and (181.19,146.57) .. (195,146.57) .. controls (208.81,146.57) and (220,157.76) .. (220,171.57) .. controls (220,185.38) and (208.81,196.57) .. (195,196.57) .. controls (181.19,196.57) and (170,185.38) .. (170,171.57) -- cycle ;
\draw   (300,171.57) .. controls (300,157.76) and (311.19,146.57) .. (325,146.57) .. controls (338.81,146.57) and (350,157.76) .. (350,171.57) .. controls (350,185.38) and (338.81,196.57) .. (325,196.57) .. controls (311.19,196.57) and (300,185.38) .. (300,171.57) -- cycle ;
\draw   (300,35.57) .. controls (300,21.76) and (311.19,10.57) .. (325,10.57) .. controls (338.81,10.57) and (350,21.76) .. (350,35.57) .. controls (350,49.38) and (338.81,60.57) .. (325,60.57) .. controls (311.19,60.57) and (300,49.38) .. (300,35.57) -- cycle ;
\draw   (439,36) .. controls (439,22.19) and (450.19,11) .. (464,11) .. controls (477.81,11) and (489,22.19) .. (489,36) .. controls (489,49.81) and (477.81,61) .. (464,61) .. controls (450.19,61) and (439,49.81) .. (439,36) -- cycle ;
\draw   (439,171.57) .. controls (439,157.76) and (450.19,146.57) .. (464,146.57) .. controls (477.81,146.57) and (489,157.76) .. (489,171.57) .. controls (489,185.38) and (477.81,196.57) .. (464,196.57) .. controls (450.19,196.57) and (439,185.38) .. (439,171.57) -- cycle ;
\draw    (220,171.57) -- (298,171.57) ;
\draw [shift={(300,171.57)}, rotate = 180] [color={rgb, 255:red, 0; green, 0; blue, 0 }  ][line width=0.75]    (10.93,-3.29) .. controls (6.95,-1.4) and (3.31,-0.3) .. (0,0) .. controls (3.31,0.3) and (6.95,1.4) .. (10.93,3.29)   ;
\draw    (439,171.57) -- (326.43,61.97) ;
\draw [shift={(325,60.57)}, rotate = 44.24] [color={rgb, 255:red, 0; green, 0; blue, 0 }  ][line width=0.75]    (10.93,-3.29) .. controls (6.95,-1.4) and (3.31,-0.3) .. (0,0) .. controls (3.31,0.3) and (6.95,1.4) .. (10.93,3.29)   ;
\draw    (350,36) -- (437,36) ;
\draw [shift={(439,36)}, rotate = 180] [color={rgb, 255:red, 0; green, 0; blue, 0 }  ][line width=0.75]    (10.93,-3.29) .. controls (6.95,-1.4) and (3.31,-0.3) .. (0,0) .. controls (3.31,0.3) and (6.95,1.4) .. (10.93,3.29)   ;https://www.overleaf.com/project/62690b01bdc14130141fb971
\draw    (195,61) -- (195,144.57) ;
\draw [shift={(195,146.57)}, rotate = 270] [color={rgb, 255:red, 0; green, 0; blue, 0 }  ][line width=0.75]    (10.93,-3.29) .. controls (6.95,-1.4) and (3.31,-0.3) .. (0,0) .. controls (3.31,0.3) and (6.95,1.4) .. (10.93,3.29)   ;
\draw    (300,35.57) -- (222,35.99) ;
\draw [shift={(220,36)}, rotate = 359.69] [color={rgb, 255:red, 0; green, 0; blue, 0 }  ][line width=0.75]    (10.93,-3.29) .. controls (6.95,-1.4) and (3.31,-0.3) .. (0,0) .. controls (3.31,0.3) and (6.95,1.4) .. (10.93,3.29)   ;
\draw    (464,61) -- (464,144.57) ;
\draw [shift={(464,146.57)}, rotate = 270] [color={rgb, 255:red, 0; green, 0; blue, 0 }  ][line width=0.75]    (10.93,-3.29) .. controls (6.95,-1.4) and (3.31,-0.3) .. (0,0) .. controls (3.31,0.3) and (6.95,1.4) .. (10.93,3.29)   ;
\draw    (325,146.57) -- (325,62.57) ;
\draw [shift={(325,60.57)}, rotate = 90] [color={rgb, 255:red, 0; green, 0; blue, 0 }  ][line width=0.75]    (10.93,-3.29) .. controls (6.95,-1.4) and (3.31,-0.3) .. (0,0) .. controls (3.31,0.3) and (6.95,1.4) .. (10.93,3.29)   ;

\draw (178.4,125) node [anchor=north west][inner sep=0.75pt]  [font=\scriptsize,rotate=-270]  {$\beta S( I+\alpha Y)$};
\draw (246,18.4) node [anchor=north west][inner sep=0.75pt]  [font=\small]  {\textcolor{red}{$\delta \varepsilon$}};
\draw (360,20.4) node [anchor=north west][inner sep=0.75pt]  [font=\scriptsize]  {$\nu \beta P( I+\alpha Y)$};
\draw (250,154.4) node [anchor=north west][inner sep=0.75pt]  [font=\small]  {$\gamma _{1}$};
\draw (320,165.4) node [anchor=north west][inner sep=0.75pt]    {$T$};
\draw (190,165.4) node [anchor=north west][inner sep=0.75pt]    {$I$};
\draw (319,29.4) node [anchor=north west][inner sep=0.75pt]    {$P$};
\draw (459,29.4) node [anchor=north west][inner sep=0.75pt]    {$Y$};
\draw (459,165.4) node [anchor=north west][inner sep=0.75pt]    {$R$};
\draw (381,97.4) node [anchor=north west][inner sep=0.75pt]  [font=\small]  {\textcolor{red}{$\delta \varepsilon$}};
\draw (470,96.4) node [anchor=north west][inner sep=0.75pt]  [font=\small]  {$\gamma _{2}$};
\draw (313,97.4) node [anchor=north west][inner sep=0.75pt]  [font=\small]  {\textcolor{red}{$\varepsilon$}};
\draw (190,30.4) node [anchor=north west][inner sep=0.75pt]    {$S$};
\end{tikzpicture}
    \caption{Flow for system \eqref{sys_simpler}. Notice that, in the $\lim_{\varepsilon\rightarrow 0}$, the SIT and PYR groups are decoupled. Indeed, the only way to go from one group to the other is through $\mathcal{O}(\varepsilon)$ passage, as highlighted by the red parameters.}
    \label{fig:flow_tot}
\end{figure}

\end{center}

Indeed, one can notice that system \eqref{sys} evolves on three distinct time scales: the fast time scale $t$, an intermediate time scale $\tau_1 = \varepsilon t$ and a slow time scale $\tau_2 = \delta \tau_1 = \delta\varepsilon t $.

Since the total population is constant, we can reduce  the dimensionality of the system from $6$ to $5$; for consistency with \cite{jardon2021geometric1,jardon2021geometric2}, we remove the $R$ compartment, substituting it via $R=1-S-I-T-P-Y$. System \eqref{sys_simpler} then becomes
\begin{align}\label{red_sys}
    S' =&\; -\beta S(I+\alpha Y)+\delta \varepsilon P, \nonumber \\
    I' =&\; \beta S(I+\alpha Y) -\gamma_1 I,  \nonumber\\
    T' =&\; \gamma_1 I -\varepsilon T,\\
    P' =&\; \delta\varepsilon + \varepsilon T(1-\delta) -\nu \beta P(I+\alpha Y)-\delta \varepsilon (S+I+2P+Y), \nonumber\\
    Y' =&\; \nu \beta P(I+\alpha Y) -\gamma_2 Y.  \nonumber
\end{align}

System \eqref{red_sys} evolves in the biologically relevant region
\begin{equation}\label{rel_region}
    \Delta := \{  (S,I,T,P,Y)\in \mathbb{R}^5\, | \, S,I,T,P,Y\geq 0,\  S+I+T+P+Y \leq 1\},
\end{equation}
as proven in the following proposition.
\begin{proposition}
The set \eqref{rel_region} is forward invariant for orbits of system \eqref{red_sys}.
\end{proposition}
\begin{proof}
It is easy to see that, for $X=S,I,T,P,Y$, one has
$$
X'|_{X=0}\geq 0.
$$
Moreover, if we write $Z=S+I+T+P+Y$, we have
$$
Z'|_{Z=1}=-\gamma_2 Y \leq 0,
$$
which was to be expected, since the only outwards flow from $Z$ is $-\gamma_2 Y$ from the $Y$ to the $R$ compartment, recall Figure \ref{fig:flow_tot}. This concludes the proof.
\end{proof}
In the following, we almost always work on the $5-$dimensional system \eqref{red_sys}; however, sometimes, e.g. in Theorem \ref{existence_EE}, it will be useful to consider the last equation of \eqref{sys_simpler} .

\section{Equilibria and 
stability}\label{sec_eqstab}

\subsection{Disease-Free Equilibrium}
The Disease-Free Equilibrium (DFE) of \eqref{red_sys}, i.e. the equilibrium in which $I=Y=0$, can be computed as $x_0 = (S_0,I_0,T_0,P_0,Y_0)=(1,0,0,0,0)$ from easy calculations. 
We now use the Next Generation Matrix method \cite{van2002repnum} to find the value of the Basic Reproduction Number, denoted by $\RO$.

\begin{proposition}\proplab{BRN}
The Basic Reproduction Number of system \eqref{red_sys} is $\RO=\frac{\beta}{\gamma_1}$.
\end{proposition}
\begin{proof}
Recall that system \eqref{red_sys} has two disease compartments, namely $I$ and $Y$. We can write
\begin{align*}
\frac{d I}{dt} &=  \mathcal{F}_1(x) - \mathcal{V}_1(x) , \\ \nonumber
     \frac{d Y}{dt} &= \mathcal{F}_2(x) - \mathcal{V}_2(x), \\
\end{align*}
where $x = (S,I,T,P,Y)$ and
\begin{align*}
\mathcal{F}_1(x) = \beta S( I + \alpha Y), &\qquad  \mathcal{V}_1(x) =\gamma_1 I,\\ \nonumber 
\mathcal{F}_2(x) = \nu \beta P( I + \alpha Y), &\qquad  \mathcal{V}_2(x) = \gamma_2 Y. \\ 
\end{align*}
Thus we obtain
\begin{equation}
F = \left( \begin{matrix} \label{mat}
\dfrac{\partial \mathcal{F}_1}{\partial I}(x_0) & \dfrac{\partial \mathcal{F}_1}{\partial Y}(x_0) \\ \\
\dfrac{\partial \mathcal{F}_2}{\partial I}(x_0) & \dfrac{\partial \mathcal{F}_2}{\partial Y}(x_0)
\end{matrix} \right) = \left( 
\begin{matrix}
\beta & \beta \alpha\\
0 & 0  \\
\end{matrix}\right) \qquad \text{and} \qquad  V = \left( \begin{matrix} 
\dfrac{\partial \mathcal{V}_1}{\partial I}(x_0) & \dfrac{\partial \mathcal{V}_1}{\partial Y}(x_0) \\ \\
\dfrac{\partial \mathcal{V}_2}{\partial I}(x_0) & \dfrac{\partial \mathcal{V}_2}{\partial Y}(x_0)
\end{matrix} \right) = \left( 
\begin{matrix}
\gamma_1 & 0\\
0 & \gamma_2  \\
\end{matrix}\right).
\end{equation}
Therefore, the next generation matrix, defined as $M := FV^{-1}$, is
\begin{equation}
    M =  \left( \begin{matrix}
    \dfrac{\beta}{\gamma_1} & \dfrac{\beta \alpha}{\gamma_2} \\ \\
    0 & 0 
    \end{matrix}\right),
\end{equation}
from which
\begin{equation}
    \RO := \rho(M) = \dfrac{\beta}{\gamma_1},
\end{equation}
where $\rho(\cdot)$ denotes the spectral radius of a matrix.
\end{proof}

\begin{remark}
Notice that the expression of the Basic Reproduction Number depends only on the first $S \rightarrow I \rightarrow T$ flow and not to the second part of the dynamics. However, as we will see later, in the fast time-scale we identify an expression of a second, ``fast'' Basic Reproduction Number, denoted by $\RO^{f}$, which depends also on the second flow $P \rightarrow Y \rightarrow R$.
\end{remark}

As a direct consequence of \propref{BRN}, we have the following lemma:

\begin{lemma}\label{lemma_DFE}
The DFE is locally asymptotically stable if $\RO <1$, and unstable if $\RO >1$.
\end{lemma}
\begin{proof}
The Jacobian matrix of \eqref{red_sys} computed in the DFE $x_0$ is
\begin{equation*}
   J_{|x_0}= \begin{pmatrix}
    0 & -\beta & 0 & \delta \varepsilon & -\beta \alpha\\
    0 & \beta - \gamma_1 & 0 & 0 & \beta \alpha\\
    0 & \gamma_1 & - \varepsilon & 0 & 0 \\
    -\delta \varepsilon & - \delta \varepsilon & \varepsilon - \delta \varepsilon & - 2 \delta \varepsilon & - \delta \varepsilon \\
    0 & 0 & 0 & 0 & - \gamma_2
    \end{pmatrix}
\end{equation*}
The eigenavalues of $J_{|x_0}$ can be easily computed as 
\begin{equation*}
    \lambda_1 = -\gamma_2, \quad \lambda_2 = -\gamma_1(1-\RO), \quad \lambda_3 = -\varepsilon, \quad \lambda_{4}=\lambda_5 = - \delta \varepsilon.
\end{equation*}
Thus, if $\RO <1$ all the eigenvalues are negative and thus the DFE is locally asymptotically stable. If, instead, $\RO >1$, $\lambda_2>0$ and the DFE loses local stability.
\end{proof}
With an additional condition on the product of the secondary infection parameters $\alpha\nu$, we are able to prove the following global stability result for the DFE.
\begin{theorem}\label{glob_stab}
Assume that $\gamma_1\leq \gamma_2$, and that $\RO < 1$. Then, the DFE is globally exponentially stable if $\alpha \nu < \dfrac{1}{\RO}.$
\end{theorem}
\begin{proof}
Clearly, by assumption $\dfrac{1}{\RO} > 1$.
Let us consider 
\begin{equation*}
    \begin{split}
    (I(t)+\alpha Y(t))' =& \beta S(t)\left(I(t)+\alpha Y(t)\right) - \gamma_1 I(t) + \alpha \nu \beta P(t) \left(I(t)+\alpha Y(t)\right) - \gamma_2 \alpha Y(t)\\
    \leq & \left(I(t)+\alpha Y(t)\right) \left( \beta (S(t) + \alpha \nu P(t)) - \gamma_1 \right) \\
    = & \beta \left(I(t)+\alpha Y(t)\right) \left(  (S(t) + \alpha \nu P(t)) - \dfrac{1}{\RO}\right).
    \end{split}
\end{equation*}
We now distinguish between two cases.

If $\alpha \nu \leq 1$, then
$$
S+\alpha \nu P \leq  S+P <1,
$$
since in the case $S+P=1$, necessarily $I+\alpha Y=0$, which implies both $I=0$ and $Y=0$, and the manifold representing absence of infection is clearly forward invariant.
Thus 
$$
 (I(t)+\alpha Y(t))' <  \beta (I(t)+\alpha Y(t)) \left( 1- \dfrac{1}{\RO}\right),
$$
and
$$
\lim_{t \to \infty} I(t)+ \alpha Y(t) = 0 \text{ exponentially, since }\RO < 1.
$$
If instead $\alpha \nu > 1$, then 
$$
S+\alpha \nu P < S + \alpha\nu (1-S) = S(1 - \alpha \nu) + \alpha \nu \leq \alpha \nu ,
$$
where the equality $S+\alpha \nu P = S + \alpha\nu (1-S)$ is not considered for the aforementioned reason. Thus  
$$
(I(t) + \alpha Y(t))' < \beta (I(t) + \alpha Y(t)) \left( \alpha \nu - \dfrac{1}{\RO} \right),
$$
and it follows that
$$
\lim_{t \to \infty} I(t)+ \alpha Y(t) = 0 \text{ exponentially if } \alpha \nu < \dfrac{1}{\RO}.
$$
On the set $\{I=Y=0\}$, $T$, $P$ and $R$ converge to $0$, and $S\rightarrow 1$. We can conclude that the DFE is exponentially stable if $\alpha \nu < \dfrac{1}{\RO}$.
\end{proof}
\begin{remark}
    As a consequence of Lemma \ref{lemma_DFE}, the DFE is always locally stable when $\RO<1$. However, even limiting ourselves to the case $\gamma_1 = \gamma_2$, when $\alpha \nu > 1$ and $1/{\alpha \nu} \le \RO < 1$, Theorem \ref{glob_stab} does not apply. Indeed, it is possible, as shown in the next Section, that in such cases there exist also endemic equilibria. In Sections \ref{sec_bifurc} and \ref{sec_num}, we will explore how the basins of attraction strongly depend on the product $\alpha \nu$, assuming all the other parameters to be fixed.
\end{remark}

\subsection{Endemic equilibria}
\label{sec:endemic}
We now discuss existence of Endemic Equilibria (EE) of system  \eqref{red_sys}, i.e. the equilibria in which $I,Y>0$.  

\begin{theorem}\label{existence_EE}\mbox{}
We distinguish the following cases.
\begin{itemize}
    \item Assume $\RO >1$. Then the system \eqref{red_sys} has a unique positive {(i.e., endemic)} equilibrium.
\item Assume $\RO<1$. Then system \eqref{red_sys} admits, for $\delta \approx 0$, two positive equilibria if and only if the following conditions hold:
\begin{equation}
\label{cond_R0.lt.1}
\left\{\begin{aligned}
    \alpha \nu >&\; \dfrac{\gamma_2}{\gamma_1 },\\
    \RO >&\; 2 \dfrac{\gamma_2}{\gamma_1 \alpha \nu} \left(\dfrac{1}{2} - \dfrac{\gamma_2}{\gamma_1 \alpha \nu} + \sqrt{ \nu - \dfrac{\gamma_2}{\gamma_1 \alpha}\left(1- \dfrac{\gamma_2}{\gamma_1 \alpha}\right) } \right).
\end{aligned}
\right. 
\end{equation}
\end{itemize}
\end{theorem}
\noindent The proof of Theorem \ref{existence_EE} can be found in Appendix \ref{app:proof}.

{For $\delta > 0$ and $\delta \not\approx 0$}, the condition on $\RO$ can be gathered from the proof, and is more cumbersome.
\begin{remark}
Clearly, when there are endemic equilibria, the DFE cannot be globally attractive. A natural question is therefore if the conditions of Theorem \ref{glob_stab} are exactly those that exclude the existence of endemic equilibria. If we consider the simple case $\gamma_1=\gamma_2$ and $\nu = 1$, Theorem \ref{glob_stab} states that the DFE is globally stable for $\RO < \min\{1, 1/\alpha\}$, while Theorem \ref{existence_EE} states that there are endemic equilibria if $\alpha > 1$ and
$$ \RO \ge \frac{1}{\alpha} + \frac{2}{\alpha}\left(\sqrt{1-\frac{1}{\alpha}+\frac{1}{\alpha^2}} - \frac{1}{\alpha}\right).$$
 Since it is not difficult to see that the right hand side  is larger than $1/\alpha$ (if $\alpha > 1$), we see there are values of $\RO$ for which we cannot either prove or disprove the global stability of the DFE. 
\end{remark}

\section{Fast-time scale} \label{sec_fast}
In this section we begin the analysis of the multiple time scales involved in the system, starting from the fast one. From here onwards, we assume to always be in the case $\RO > 1$.

System \eqref{red_sys} is a slow-fast system written in the non-standard form of GSPT; that can be seen by writing it as
\begin{align*}
    z ' = H(z) + \varepsilon G(z; \varepsilon, \delta)
\end{align*}
where
\begin{gather*}
    z = \begin{pmatrix}
    S\\  I\\ T\\ P\\ Y
    \end{pmatrix}, 
    \quad 
    H(z) =  \begin{pmatrix}
     -\beta S(I+\alpha Y)  \\
     \beta S(I+\alpha Y) -\gamma_1 I  \\
     \gamma_1 I \\
     -\nu \beta P(I+\alpha Y)\\
    \nu \beta P(I+\alpha Y) -\gamma_2 Y
     \end{pmatrix},
     \quad 
     G(z; \varepsilon, \delta)=
     \begin{pmatrix}
     \delta P  \\
     0 \\
     - T\\
     \delta +  T(1-\delta) -\delta (S+I+2P+Y)\\
    0
     \end{pmatrix}.
\end{gather*}
see \cite{fenichel1979geometric, wechselberger2020geometric} for details.

We take $\varepsilon \rightarrow 0$ in system \eqref{red_sys}, obtaining the fast limit system
\begin{equation}
 \label{sys_fast}   
\begin{split}
    S' =&\; -\beta S(I+\alpha Y),\\
    I' =&\; \beta S(I+\alpha Y) -\gamma_1 I, \\
    T' =&\; \gamma_1 I, \\
    P' =&\; -\nu \beta P(I+\alpha Y),\\
    Y' =&\; \nu \beta P(I+\alpha Y) -\gamma_2 Y.
\end{split}
\end{equation}
The $1$-critical manifold is defined as the set of equilibria of \eqref{sys_fast}, i.e. it is given by the set
\begin{equation}\label{crit_manif}
    \mathcal{C}_{1} := \{ (S,I,T,P,Y)\in \mathbb{R}^5 \, | \, I=Y= 0 \}.
\end{equation}

We emphasise that, since we are analysing a system which evolves on three time scales, we adopt the notation used in \cite{cardin2017fenichel,kaklamanos2022bifurcations} and call the critical manifold in the fast-intermediate time scales $1$-critical, and the one in the intermediate-slow $2$-critical. There is no flow anymore from first three compartments ($S$, $I$, $T$) and last two ($P$, $Y$), although the two resulting subsystems are not decoupled, due to $I$ and $Y$ still playing a role in both.

In the following, for a given solution of \eqref{sys_fast}, we denote by
\begin{align}
X_\infty =\lim_{t \to +\infty} X(t), \quad X = S, I, T, P, Y, \label{Xinf}
\end{align}
the limit of the corresponding variable on $\mathcal{C}_1$ , where $X_0$ denotes the value at $t=0$.

\begin{proposition}
Trajectories of system \eqref{sys_fast} converge to $\mathcal{C}_1$ \eqref{crit_manif} as $t \to +\infty$. 
\end{proposition}

\begin{proof}
We apply strategies similar to  \cite{andreasen2011final,brauer2008epidemic}.
Recall that trajectories of system, and hence of system \eqref{sys_fast}, evolve on the compact set $\Delta$, defined in \eqref{rel_region}. 
Since $S' < 0$, there exists $S_\infty \ge 0$; similarly from $S'+I'<0$, there exists $(S+I)_\infty$, hence $I_\infty \ge 0$.

Now, integrating $S'+I'$ from system \eqref{sys_fast}, we obtain
$$
-\infty< S_\infty +I_\infty -S_0-I_0=\int_0^{+\infty}(S'(s)+I'(s))\text{d}s=-\gamma_1\int_0^{+\infty} I(s)\text{d}s <0,
$$
hence $I_\infty=0$. Similarly, by integrating $P'+Y'<0$, we can conclude that $Y_\infty=0$.

Hence, we have
$$
S_\infty= S_0 \exp \left( - \beta \int_0^\infty (I(t)+\alpha Y(t))\text{d}t \right)>0,$$
similarly, $P\rightarrow P_\infty>0$ and $T\rightarrow T_0 +S_0-S_\infty$. 
\end{proof}
The values of $S_\infty$ and $P_\infty$ can actually be computed, as in \cite{andreasen2011final}. We do so in the following Lemma.
\begin{lemma}
The limit value under the fast flow \eqref{sys_fast} $S_\infty$ is the unique solution in $(0,S_0)$ of 
\begin{equation}
    \label{conserv2}
\log \left( \frac{S_\infty}{S_0}\right) - \beta  \left(\dfrac{S_\infty -S_0 }{\gamma_1} + \alpha \dfrac{P_0 \left( \frac{S_\infty}{S_0}\right)^\nu -P_0 }{\gamma_2}\right) = - \beta  \left(\dfrac{ I_0}{\gamma_1} + \alpha \dfrac{ Y_0}{\gamma_2}\right),
 \end{equation}
whereas $P_\infty$ is obtained as
\begin{equation}
    \label{conserv1}
    P_\infty = P_0 \left( \frac{S_\infty}{S_0}\right)^\nu.
 \end{equation}
 \end{lemma}
\begin{proof}
 Indeed
\begin{equation*}
   \nu \frac{\text{d}}{\text{d}t} \log(S(t)) = \frac{\text{d}}{\text{d}t} \log(P(t)) \implies \left( \frac{S(t)}{S_0}\right)^\nu = \frac{P(t)}{P_0}.
\end{equation*}
Hence, taking the limit as $t \to +\infty$, we obtain \eqref{conserv1}.

Furthermore
\begin{equation*}
   \beta \frac{\text{d}}{\text{d}t} \left(\dfrac{S(t)+I(t)}{\gamma_1} + \alpha \dfrac{P(t)+Y(t)}{\gamma_2}\right)= - \beta(I(t) + \alpha Y(t)) = \frac{\text{d}}{\text{d}t} \log(S(t)) ,
\end{equation*}
which implies
\[
\log \left( \frac{S(t)}{S_0}\right) = \beta \left(\dfrac{S(t)+I(t)-S_0 - I_0}{\gamma_1} + \alpha \dfrac{P(t)+Y(t)-P_0 - Y_0}{\gamma_2}\right).
\]
Taking the limit as $t\rightarrow+\infty$, recalling that the solutions converge to the manifold $\mathcal{C}_1$, and using \eqref{conserv1} we get \eqref{conserv2}.

To show the uniqueness (known from the general result by Andreasen \cite{andreasen2011final}), 
we introduce
\begin{align*}
L(x) = \log \left( \frac{x}{S_0}\right) - \beta  \left(\dfrac{x -S_0 }{\gamma_1} + \alpha \dfrac{P_0 \left( \frac{x}{S_0}\right)^\nu -P_0 }{\gamma_2}\right).
\end{align*}
Then,
\begin{equation}
    \label{Hder}
L'(x) = \frac1x - \frac{\beta}{\gamma_1} - \beta  \alpha \dfrac{P_0 \nu x^{\nu -1}}{\gamma_2 S_0^\nu} =: \frac{1}{x} l(x),
 \end{equation}
where
\begin{equation}
    \label{hx}
    l(x) = 1 - \frac{\beta}{\gamma_1} x - \frac{\beta}{\gamma_2}  \alpha \nu P_0 \left(\dfrac{x }{S_0}\right)^\nu.
\end{equation}
It is clear that $l$ is a decreasing function, hence it has a unique 0. From \eqref{Hder}, we see that $L$ has a unique extremum in $x>0$ that has to be a maximum. From $L(0_+) = - \infty$ and $L(S_0) = 0 $ we conclude that \eqref{conserv2} has a unique solution in $(0,S_0)$.
\end{proof}

Notice that 
$$L'(S_0) = \dfrac{1}{S_0} \left( 1 - \frac{\beta}{\gamma_1} S_0 - \frac{\beta}{\gamma_2}  \alpha \nu P_0 \right),$$
hence 
\begin{equation}
    \label{R0_fast}
    L'(S_0) \lessgtr 0 \iff \RO^{f} := \frac{\beta}{\gamma_1} S_0 + \frac{\beta}{\gamma_2}  \alpha \nu P_0 \gtrless 1.
\end{equation}
Therefore, the equation $L(S_\infty) = 0$ has a unique solution in $(0,S_0)$ if $\RO^f > 1$, while it has no solutions in $(0,S_0)$ if $\RO^f \le 1$.

This means that, if we consider the limiting case of \eqref{sys_fast} with $I_0,\ Y_0 \approx 0$ (which is the typical case both when a new pathogen is introduced in a population and, as we will see, when the system re-enters the fast time scale), we have $S_\infty \ll S_0$ (thus an epidemic) only if $\RO^f > 1$.

We can therefore limit ourselves to consider the case in which $\RO^f > 1$. After the fast flow, solutions ``land'' on the attracting part of the critical manifold $\mathcal{C}_1$, meaning
 \begin{equation}\label{R0veloce}
     \frac{\beta}{\gamma_1} S_\infty + \frac{\beta}{\gamma_2}  \alpha \nu P_\infty  < 1 ,
\end{equation}
as we show in the following section.

\subsection{Eigenvalues on the 1-critical manifold}

We can re-write the $I,Y$ equations representing infected individuals of the fast system \eqref{sys_fast} in vector form:
$$
\begin{pmatrix}
I\\
Y
\end{pmatrix}'=
\begin{pmatrix}
\beta  S-\gamma_1 & \alpha\beta S\\
\nu \beta P & \alpha \nu \beta P-\gamma_2 
\end{pmatrix}
\begin{pmatrix}
I\\
Y
\end{pmatrix}=:A
\begin{pmatrix}
I\\
Y
\end{pmatrix}.
$$
We are only interested in these variables, since the eigenvalues associated to the remaining three variables will all be $0$ on the $1$-critical manifold, since we have already taken the limit $\varepsilon\rightarrow 0$ \cite{wechselberger2020geometric}. 
To compute the eigenvalues of the matrix $A$, we observe that the characteristic equation of $A$ is as follows:
\begin{equation}\label{eq:lambdaquadr}
    \lambda^2 + \lambda(\gamma_1 + \gamma_2 - \alpha\nu\beta P - \beta S) + \gamma_1\gamma_2 - \alpha\nu\beta\gamma_1 P - \beta \gamma_2 P =0.
\end{equation}
To make the analysis  less cumbersome, we will assume  in what follows $\gamma_1= \gamma_2 =: \gamma$. Under this assumption, equation \eqref{eq:lambdaquadr} reduces to
\begin{equation*}
    \lambda^2 + \lambda(2 \gamma - \alpha\nu\beta P - \beta S) + \gamma^2 - \alpha\nu\beta\gamma P - \beta \gamma P =0,
\end{equation*}
and thus
\begin{equation*}
    \lambda_{1,2} = \dfrac{\beta(\alpha\nu P + S) - 2\gamma \pm  \beta(\alpha\nu P + S)}{2},
\end{equation*}
from which we obtain the two eigenvalues
\begin{equation*}
    \lambda_1 = -\gamma \;\; \text{and} \;\; \lambda_2 = -\gamma +  \beta(\alpha\nu P + S).
\end{equation*}

Notice that $\lambda_1 < \lambda_2$ for $S,P> 0$. This separation in eigenvalues is one of the crucial assumptions to apply the entry-exit function, as we will see in Section \ref{sec:entrex}.

\section{Intermediate time scale}\label{sec_inter}

In this section, we analyse the evolution of system \eqref{red_sys} on the intermediate time scale $\tau_1=\varepsilon t$. 

Consider \eqref{red_sys}, and assume that a solution reached an $\mathcal{O}(\varepsilon)$ neighbourhood of the $1$-critical manifold $\mathcal{C}_1$ \eqref{crit_manif}. We rescale the infectious compartment by $I=\varepsilon m$, $Y=\varepsilon n$, and obtain
\begin{align*}
    S' =&\; -\varepsilon\beta S(m+\alpha n)+\delta \varepsilon P, \nonumber \\
     m' =&\; \beta S(m+\alpha n) -\gamma m,  \nonumber\\
    T' =&\; \varepsilon\gamma m -\varepsilon T,\\
    P' =&\;\varepsilon\delta+ \varepsilon T(1-\delta) -\varepsilon\nu \beta P(m+\alpha n)-\delta \varepsilon (S+\varepsilon m+2P+\varepsilon n), \nonumber\\
 n' =&\; \nu \beta P(m+\alpha n) -\gamma n. \nonumber  
\end{align*}
We then apply a rescaling to the time coordinate, bringing the system to the intermediate time scale $\tau_1=\varepsilon t$:
\begin{align*}
    S\intr =&\; -\beta S(m+\alpha n)+\delta P, \nonumber \\
     \varepsilon m\intr =&\; \beta S(m+\alpha n) -\gamma m,  \nonumber\\
    T\intr =&\; \gamma m - T,\\
    P\intr =&\; \delta+  T(1-\delta) -\nu \beta P(m+\alpha n)-\delta  (S+\varepsilon m+2P+\varepsilon n), \nonumber\\
 \varepsilon n\intr =&\; \nu \beta P(m+\alpha n) -\gamma n, \nonumber
\end{align*}
letting $\intr$ denotes the derivative with respect to the intermediate time scale $\tau_1$. 

If we look at these equations on the $1$-critical manifold, now determined by $m=n=0$, we obtain the linear subsystem
\begin{align*}
    S\intr =&\; \delta P, \nonumber \\
     T\intr =&\;  - T,\\
    P\intr =&\;\delta+ T (1-\delta) -\delta(S+2  P), \nonumber
\end{align*}
with $S$ now being $\delta$-slow, $P$, $T$ fast, and eigenvalues $0$, $-1$, $-\delta$.

We now take $\delta \rightarrow 0$:
\begin{align*}
    T\intr =&\; -T, \\
    P\intr =&\; T.
\end{align*}
The $2$-critical manifold is then
\begin{equation}\label{eqn:subcrit}
    \mathcal{C}_2:=\{(S,I,T,P,Y)\in \mathcal{C}_{1} \, | \, T=0\}.
\end{equation} 
In the intermediate time scale, we have $T\rightarrow 0$, $P\rightarrow P_\infty+T_0 +S_0-S_\infty= P_\infty + T_\infty$. Nothing else happens, meaning $I$ and $Y$ (or, equivalently, $m$ and $n$) remain 0, whereas $S$ remains at its values from the fast time scale. Instead, $T$ and $P$ evolve according to the following formulas:
\begin{equation}\label{interm_TP}
\begin{aligned}
    T(\tau_1) =&\; T_\infty e^{-\tau_1}, \\
    P(\tau_1) =&\; P_\infty +T_\infty(1-e^{-\tau_1}).
\end{aligned}
\end{equation}
In the following section, we will investigate the delayed loss of stability of the $1$-critical manifold, applying the so-called \emph{entry-exit function}.

\subsection{Entry-exit function}\label{sec:entrex}

The entry-exit function is an important tool in the field of multiple time scale systems in which the critical manifold is not uniformly hyperbolic, and thus standard GSPT theory can not be applied. 

More specifically, in its lowest dimensional formulation, this construction applies to planar systems of the form 
\begin{equation}\label{eq:entex}
\begin{aligned}
x'&= f(x,y,\varepsilon)x,\\
y'&=\varepsilon g(x,y,\varepsilon),
\end{aligned}
\end{equation}
with $(x,y)\in \mathbb{R}^2$, $g(0,y,0)>0$ and $\textnormal{sign}(f(0,y,0))=\textnormal{sign}(y)$. Note that for $\varepsilon=0$, the $y$-axis consists of normally attracting/repelling equilibria if $y$ is negative/positive, respectively.
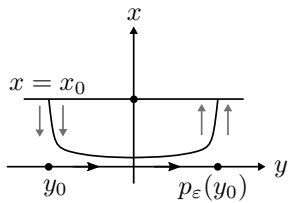
\begin{figure}[htbp]\centering
	\begin{tikzpicture}
		\node at (0,0){\includegraphics[scale=0.85]{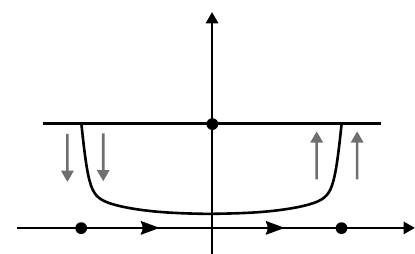}};
		\node at (2,-.9){$y$};
		\node at (0.05,1.15){$x$};
		\node at (-1.1,.2){$x=x_0$};
		\node at (-1,-1.15) {$y_0$};
		\node at ( 1.1,-1.15) {$p_\varepsilon(y_0)$};
	\end{tikzpicture}
	\caption{Visualization of the entry-exit map on the line $x=x_0$.} 
	\label{fig:entrex2D}
\end{figure}

Consider a horizontal line $\{x=x_0\}$, close enough to the $y$-axis to obey the attraction/repulsion assumed above. An orbit of \eqref{eq:entex} that intersects such a line at $y=y_0<0$ (entry) re-intersects it again (exit) at $y=p_\varepsilon(y_0)$, as sketched in Figure~\ref{fig:entrex2D}.

As $\varepsilon \rightarrow 0$, the image of the return map $p_\varepsilon(y_0)$ to the horizontal line $x=x_0$ approaches $p_0(y_0)$ given implicitly by
\begin{equation}\label{eq:pzero}
\int_{y_0}^{p_0(y_0)} \frac{f(0,y,0)}{g(0,y,0)}\textnormal{d}y = 0.
\end{equation}
This construction can be generalized to higher dimensional systems, such as the one we are studying in this paper. 
For a more precise description of the planar case, we refer to \cite{de2008smoothness,de2016entry} or the preliminaries of \cite{jardon2021geometric1}. For more general theorems, we refer the interested reader to \cite{liu2000exchange,neishtadt1987persistence,neishtadt1988persistence,schecter2008exchange}.

Note than in many cases it is only possible to compute the exit \emph{time} $\tau_E$, rather than an exit \emph{point} \cite{jardon2021geometric1,jardon2021geometric2}. Without delving too far in the precise details, assume that the system is still \eqref{eq:entex}, but with $(x,y) \in \R^{m+n}$ and $f(x,y,\varepsilon)$ an $m \times m$ matrix. Assume that only one eigenvalue (let it $\lambda_1(s)$) of $f(0,y(s;y_0),0)$  changes its sign as $s$ increases from 0 to $\infty$; assume, moreover, that this eigenvalue is separated from all the other eigenvalues for all values of $s \in \mathbb{R}$. The exit \emph{time} $\tau_E$ in the slow time scale can be, under some additional conditions, obtained through
$$ \int_0^{\tau_E} \lambda_1(s) \, ds = 0.$$
Here $y(s;y_0)$ is the solution of 
$$ \dot y = g(0,y,0),\;\; y(0) = y_0,$$
and $\lambda_1(s)$ is the principal eigenvalue of $f(0,y,0)$ computed at $y= y(s;y_0)$.

We remark that a recent result was achieved in weakening the assumption of the eigenvalue separation, providing a generalization of the known entry-exit formulae \cite{kaklamanos2022entry}.

In this section, we give conditions on $S_\infty$, $P_\infty$ and $T_\infty$ (recall eq.\ \eqref{Xinf}), to determine whether an orbit approaches the $2$-critical manifold $\{ I=Y=T=0 \}$, entering the slow time scale, or the system returns to the fast scale from the intermediate one.

\begin{proposition}\label{int_or_slow}
If
\begin{equation}
\label{exit_inter}\beta \left(S_\infty + \alpha \nu (P_\infty+ T_\infty)\right)>\gamma,
\end{equation}
then, for $\varepsilon$ and $\delta$ sufficiently small, the entry-exit phenomenon happens on the intermediate scale, i.e. the orbit will not reach a $\mathcal{O}(\delta)$ neighbourhood of the $2$-critical manifold $\{ I=Y=T=0 \}$.
If, on the other hand, 
\begin{align}
\beta \left(S_\infty + \alpha \nu (P_\infty+ T_\infty)\right)\le \gamma,
\end{align}
then the corresponding orbit enters a $\mathcal{O}(\delta)$ neighbourhood of the $2$-critical manifold $\{ I=Y=T=0 \}$, and the system enters the slow flow.
\end{proposition}
\begin{proof}
Recall that the eigenvalue which gives the change of stability of the $1$-critical manifold $\{ I=Y=0 \}$ is $\lambda_1 = -\gamma +  \beta(\alpha\nu P + S)$. We therefore need to check if $\int_0^x \lambda_1(s) \text{d}s=0$ has solutions $x > 0$. Using the explicit expressions with $S$ constant and $P$ evolving according to \eqref{interm_TP},
this implies studying the existence of a value $x >0$ such that 
$$ \int_0^x (-\gamma +\beta S_\infty +\beta \alpha \nu P_\infty+\beta \alpha \nu T_\infty ( 1-  e^{-s})) \text{d}s = 0.$$
It is convenient to divide this expression by $x$ and, computing explicitly the integral, define
\begin{equation}
    \f(x)= \begin{cases}
        -\gamma +\beta S_\infty +\beta \alpha \nu P_\infty+\beta \alpha \nu T_\infty - \beta \alpha \nu T_\infty \dfrac{(1-e^{-x})}{x} &\mbox{if } x > 0 \\
        -\gamma +\beta S_\infty +\beta \alpha \nu P_\infty &\mbox{if } x = 0.
    \end{cases}
\label{equat_exit_ind}
\end{equation}
We are thus looking for a positive root of $\f(x) = 0$. It is immediate to see that $\f$ is a continuous increasing function. Furthermore from Section \ref{sec_fast} we know that $\RO(S_\infty + \alpha \nu P_\infty) < 1 $, i.e.\ $\f(0) < 0$. Finally, we have
$$
\lim_{x\rightarrow +\infty} \f(x)=
-\gamma +\beta S_\infty +\beta \alpha \nu P_\infty+\beta \alpha \nu T_\infty.
$$
Hence, if \eqref{exit_inter} holds,
there exists a unique positive root for $\f$. This implies that, for $\varepsilon$ and $\delta$ small enough, the entry-exit phenomenon happens already on the intermediate scale, and the orbit will not reach a $\mathcal{O}(\delta)$ neighbourhood of the $2$-critical manifold $\{ I=Y=T=0 \}$.

If instead, $\beta (S_\infty + \alpha \nu (P_\infty+ T_\infty)\le \gamma$, the equation $\f(x)=0$ has no positive solutions; this implies that the corresponding orbit eventually approaches a $\mathcal{O}(\delta)$ neighbourhood of the $2$-critical manifold $\{ I=Y=T=0 \}$, and the system enters the slow flow.
\end{proof}
The two cases {presented in Proposition \ref{int_or_slow}} are illustrated in Figures \ref{cond_less} and \ref{cond_greater}, respectively.

\begin{figure}[H]
     \centering
     \begin{subfigure}[b]{0.49\textwidth}
         \centering
         \includegraphics[width=\textwidth]{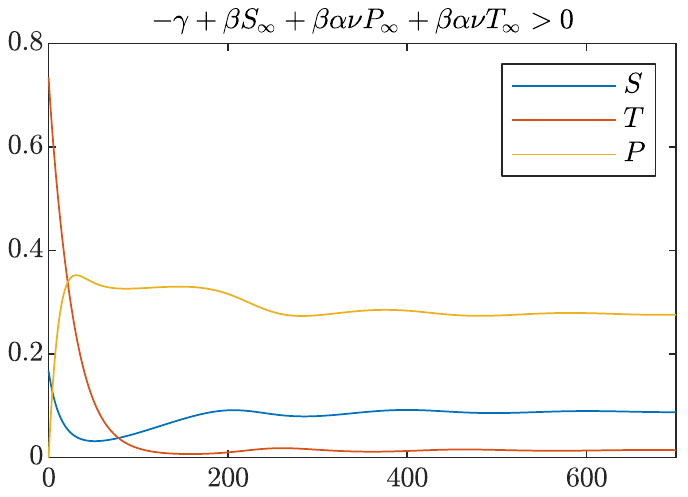}
         \caption{}
     \end{subfigure}
     \hfill
     \begin{subfigure}[b]{0.49\textwidth}
         \centering
         \includegraphics[width=\textwidth]{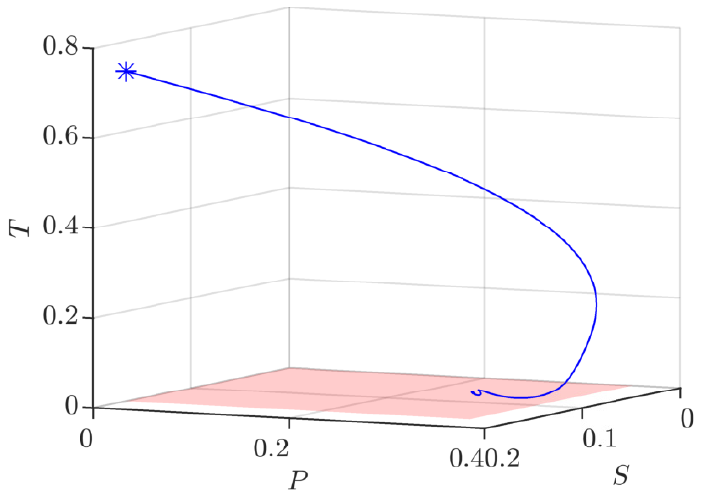}
         \caption{}
     \end{subfigure}
     \caption{A simulation of \eqref{red_sys} starting from values $(S,T,P) = (0.1667,0,0)$; panel a) shows $S$, $P$ and $T$ vs.\ time; panel b) the solution in the 3-d phase space. We name the initial values $S_\infty$, $P_\infty$ and $T_\infty$, since they satisfy $\beta S_\infty +\beta \alpha \nu P_\infty > \gamma$ and could be values reached at the end of the fast time scale with $I = Y \approx 0$. The values of the parameters are $\beta = 0.9$, $\alpha = 0.5$, $\nu = 0.7$, $\gamma_1 = \gamma_2 = 1/6$, $\delta = 1/20$, $\varepsilon = 1/20$.}
     \label{cond_less}
\end{figure}

\begin{figure}[H]
     \centering
     \begin{subfigure}[b]{0.49\textwidth}
         \centering
\includegraphics[width=\textwidth]{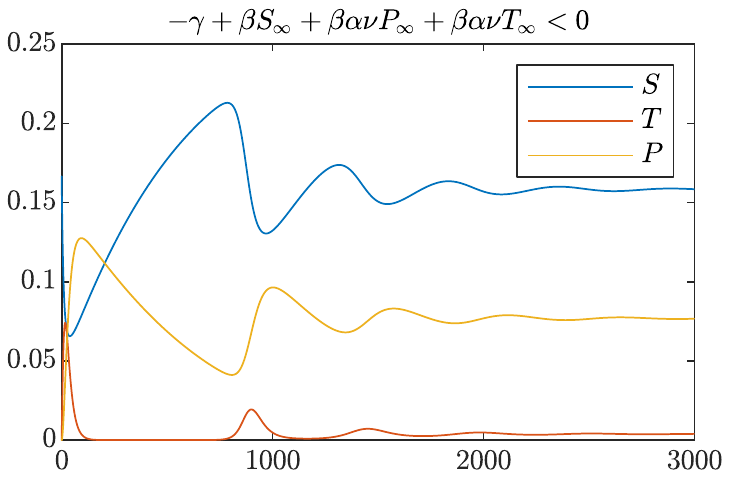}
         \caption{}
     \end{subfigure}
     \hfill
     \begin{subfigure}[b]{0.49\textwidth}
         \centering
\includegraphics[width=\textwidth]{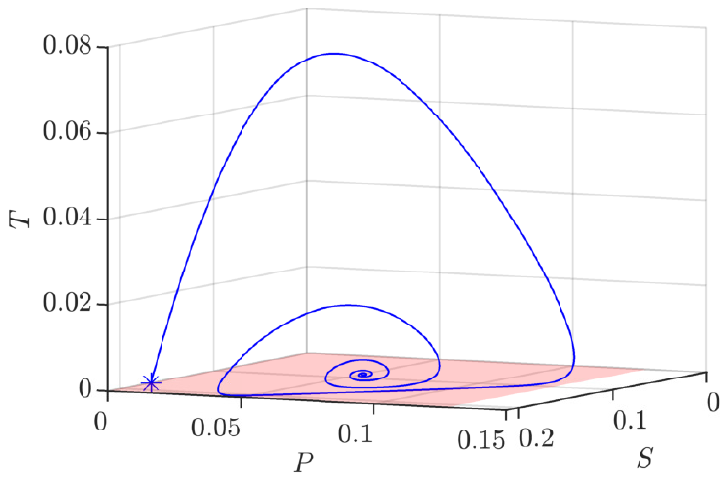}
         \caption{}
     \end{subfigure}
     \caption{Same parameters as in Figure \ref{cond_less}, except that here the initial values are $(S,T,P) = (0.1667, 0.7333, 0)$.}
     \label{cond_greater}
\end{figure}

\section{Slow time scale}\label{sec_slow}

Assume now that the intermediate scale does not lead to an exit from the $1$-critical manifold. Then, $T\rightarrow 0$, and we apply a rescaling once the dynamics arrive $\mathcal{O}(\delta\varepsilon)$-close to the $2$-critical manifold $\mathcal{C}_2$ \eqref{eqn:subcrit}.

We introduce new variables $\delta\varepsilon v= I$, $\delta\varepsilon w= Y$ and $\delta u = T$, system \eqref{red_sys} becomes

\begin{align}
   S' =&\; \delta\varepsilon \left(-\beta  S(v+\alpha w)+P\right), \nonumber \\
   \delta\varepsilon v' =&\; \delta\varepsilon \beta  S(v+\alpha w) -\delta\varepsilon\gamma_1 v,  \nonumber\\
    \delta u'=&\; \delta\varepsilon ( \gamma_1 v -  u),\\
    P' =&\;  \delta\varepsilon(1 +  u (1-\delta) -\nu \beta P (v +\alpha w) - (S+\delta\varepsilon v +2P +\delta\varepsilon w)) , \nonumber\\
  \delta\varepsilon  w' =&\; \delta\varepsilon\nu \beta P(v+\alpha w) -\delta\varepsilon\gamma_2 w, \nonumber
\end{align}
with $S,P$ slow, $v,w$ fast and $u$ intermediate. Rescaling to the slow time variable $\tau_2 = \delta\varepsilon t$, we obtain
\begin{align}
   \Dot{S} =&\;  \left(-\beta  S(v+\alpha w)+P\right), \nonumber \\
   \delta\varepsilon  \Dot{v} =&\;  \beta  S(v+\alpha w) -\gamma_1 v,  \nonumber\\
    \varepsilon  \Dot{u} =&\;  (\varepsilon \gamma_1 v -  u),\\
     \Dot{P} =&\;  1 +  u (1-\delta) -\nu \beta P (v +\alpha w) - (S+\delta\varepsilon v +2P +\delta\varepsilon w) , \nonumber\\
  \delta\varepsilon   \Dot{w} =&\; \nu \beta P(v+\alpha w) -\gamma_2 w, \nonumber
\end{align}

and where now the overdot represents the derivative with respect to the slow time variable $\tau_2$.

The evolution of $S$ and $P$ on $2$-critical manifold \eqref{eqn:subcrit} is dictated by the following ODEs:
\begin{align}\label{eqn:slowsl}
    \Dot{S} =&\;  P, \nonumber \\
    \Dot{P} =&\; 1 - S- 2P. 
\end{align}
Hence, in this time scale, $S$ increases, whereas $P\rightarrow 0$, as we will see shortly. 

The linear ODEs \eqref{eqn:slowsl} can be solved explicitly, recalling that at the beginning of the slow flow $(S,P)=(S_\infty,P_\infty+T_\infty)$.
 It is more convenient to perform the computations by including the variable $R$, so that
\begin{align}
\label{sol_slow}
    R(\tau_2) =&\; R_\infty e^{-\tau_2}, \nonumber\\
    P(\tau_2) =&\; (P_\infty+T_\infty) e^{-\tau_2} +\tau_2 R_\infty e^{-\tau_2}  = (1-S_\infty)\tau_2  e^{-\tau_2} + (P_\infty+T_\infty)(1-\tau_2 ) e^{-\tau_2},
    \\
    S(\tau_2) =&\; 1 - (P_\infty+T_\infty) e^{-\tau_2} - R_\infty (1 + \tau_2) e^{-\tau_2} = 1 -  (1-S_\infty)(1+\tau_2 ) e^{-\tau_2}+(P_\infty+T_\infty) \tau_2  e^{-\tau_2}.\nonumber
\end{align}
As a consequence of analysis carried out in Section \ref{sec:entrex}, the exit time $T_E$ satisfies the equation
\begin{equation}
    \label{entry_exit}
    - \gamma T_E + \beta \int_0^{T_E} (S(x) + \alpha \nu P(x))\, dx = 0.
\end{equation}
To elaborate further on this formula, notice that if the dynamics reach a neighbourhood of $\mathcal{C}_2$ in finite time, then the time it took to get there is $\mathcal{O}(\delta)$ with respect to the time it will spend close to $\mathcal{C}_2$. Hence, assuming $\delta$ is small enough, we can ignore the intermediate time scale when computing the exit time, see Figure \ref{fig:entrex} for a visualization.

Since the $2$-critical manifold $\mathcal{C}_2$ \eqref{eqn:subcrit} does not lose stability as part of the $1$-critical manifold $\mathcal{C}_1$ \eqref{crit_manif}, orbits may leave the $2$-critical manifold only when simultaneously leaving the $1$-critical manifold as well. 

Moreover, since $S\rightarrow 1$ on the slow time scale and $\RO>1$ means $\beta>\gamma$, the eigenvalue which provides the change of stability of the $1$-critical manifold $\lambda_2$ will eventually become and remain positive under the slow flow, ensuring an exit.

\begin{figure}[H]
    \centering
    \begin{tikzpicture}
   \node at (0,0){\includegraphics[width=0.6\textwidth]{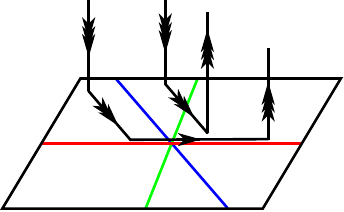}};
   	\node at (3,-3) {$\mathcal{C}_1$};
   	\node at (3,-1.5) {\textcolor{red}{$\mathcal{C}_2$}};
	\node at (-2.5,-0.5) {$\mathcal{O}(1/\varepsilon)$};
	\node at (1,-1.35) {$\mathcal{O}(1/\varepsilon\delta)$};
	\end{tikzpicture}
    \caption{Black plane: $1$-critical manifold $\mathcal{C}_1=\{ I=Y=0\}$ {(a 3D manifold in $\R^5$)}; red line: $2$-critical manifold $\mathcal{C}_2=\{ I=Y=T=0\}$ {(a 2D manifold in $\R^5$)}; green line: loss of hyperbolicity line, where $\lambda_2=0$; {blue line: above, condition \eqref{exit_inter} is satisfied, below it is not}. We distinguish between two cases: in the first, the fast flow (triple arrows) lands on $\mathcal{C}_1$, approaches $\mathcal{C}_2$ in a time $t=\mathcal{O}(1/\varepsilon)$ (intermediate flow, double arrows), then flows along $\mathcal{C}_2$ for a time $t=\mathcal{O}(1/\varepsilon\delta)$ (slow flow, single arrow), and finally exits a neighbourhood of the $1$-critical manifold. In the second one, the fast flow lands close to the green line, and the entry-exit happens already on the intermediate time scale (see Section \ref{sec:entrex} and \eqref{entry_exit} for details).}
    \label{fig:entrex}
\end{figure}
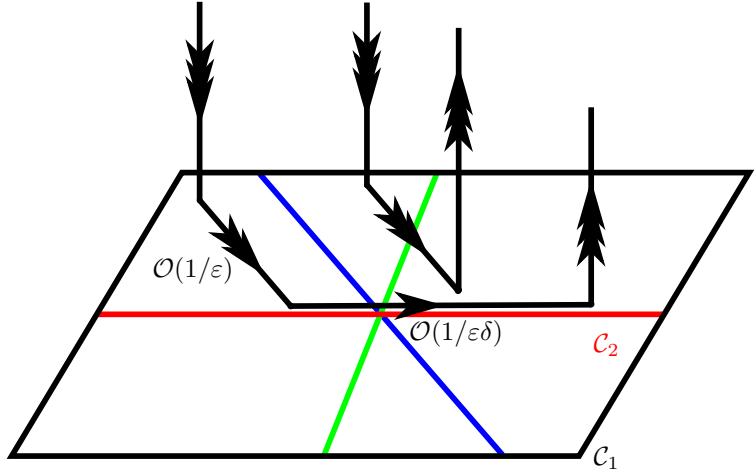
\section{The system as a sequence of discrete maps}\label{sec_map}
We can summarize the behaviour of the system for $\varepsilon, \delta \approx 0$ through two maps, the first one describing the fast scale, the second one either the intermediate (in case the systems exits from there to the fast scale) or the intermediate plus slow scales (otherwise).

We assume that the system starts the fast scale at values of $I_0, Y_0 \approx 0$ and with values of $S_0$ and $P_0$ such that $\RO^f > 1$ (recall \eqref{R0_fast}). This condition can be usefully rewritten as $(S_0,P_0) \in \Lambda_+$ by introducing the function
\begin{equation}
    \label{linear_map}
    L(S,P)  := \frac{\beta}{\gamma} \left( S + \alpha \nu P \right) - 1,
\end{equation}
and the sets
\begin{align*} \Lambda_+ &= \{(S,P) \in \R^2_+ : S+P \le 1,\ L(S,P) > 0\}, \\
 \Lambda_- &= \{(S,P) \in \R^2_+ : S+P \le 1,\ L(S,P) < 0\} \\ \mbox{ and } \quad  \Lambda_0 &= \{(S,P) \in \R^2_+ : S+P \le 1,\ L(S,P) = 0\}.
 \end{align*}
Under those conditions, the fast system converges to a point  $(S_\infty, P_\infty) \in \Lambda_-$; these values can be obtained by solving $H(S_\infty) = 0$ (see \eqref{conserv2}) and using \eqref{conserv1}. 

First of all, we denote by $F=(F_1,F_2)$ this map from $\Lambda_+$ into $\Lambda_-$. In formulae, we define $F_1(S,P)$ equal to the smallest positive root of $H(S)=0$, with $H$ defined in \eqref{conserv2}, while 
\begin{equation}\label{F2def}
F_2(S,P) = P \left( \frac{F_1(S,P)}{S}\right)^\nu + S - F_1(S,P). \end{equation}
Although the map $F$ is defined only in $\Lambda_+$, we can extend it with continuity to $\Lambda_0$ obtaining that all points of $ \Lambda_0$ are fixed points.

Furthermore, during the fast phase $T$ reaches the value $T_\infty = T_0 + S_0 - S_\infty$.

To include this third variable in the discrete map, we extend the sets $\Lambda_+$, $\Lambda_-$ and $\Lambda_0$ to
$$ \tilde\Lambda_p =\{(S,P,T)\in \R^3_+: S+P+T \le 1,\ (S,P) \in \Lambda_p\} $$
with $p = +,\ -$ or $0$. 

Then we consider the map $\tilde F=(\tilde F_1,\tilde F_2,\tilde F_3)$ from $ \tilde \Lambda_+$ into $\tilde \Lambda_-$ defined through
\begin{equation}
    \tilde F_1(S,P,T)=F_1(S,P) \quad \tilde F_2(S,P,T)=F_2(S,P) \quad \tilde F_3(S,P,T)=T+S-F_1(S,P).
\label{def_F_tilde} 
\end{equation}

During the intermediate time-scale, $P$ would increase towards the value $P_\infty + T_\infty$, while $S$ does not change. There are two possibilities, as noticed in Section \ref{sec:entrex}: either the system eventually re-enters the region $\Lambda_+$ and exits the $1$-critical manifold $\mathcal{C}_1$ at a time given by the entry-exit map; or the point $(S_\infty,P_\infty + T_\infty ) \in \Lambda_- \cup \Lambda_0$, in which case the system will reach the $2$-critical manifold $\mathcal{C}_2$. The first case occurs instead when $(S_\infty,P_\infty + T_\infty ) \in \Lambda_+$, which is equivalent to  \eqref{exit_inter}.

If $(S_\infty,P_\infty,T_\infty)=\tilde F(S_0,P_0,T_0)$ satisfies \eqref{exit_inter}, then we can define $\tilde G(S_\infty,P_\infty,T_\infty)$ through the entry-exit map defined in Section \ref{sec_inter}. Precisely, the exit time $t_E$ will be the root of $\f(t_E) = 0$ where $\f$ is defined in \eqref{equat_exit_ind}.

Then 
\begin{equation}
    \label{G_case1}
    \tilde G_1(S_\infty,P_\infty,T_\infty)= S_\infty\quad \tilde G_2(S_\infty,P_\infty,T_\infty)=P_\infty+T_\infty(1-e^{-t_E})\quad \tilde G_3(S_\infty,P_\infty,T_\infty)=T_\infty e^{-t_E}.
\end{equation}

Note that the map $\tilde G$ can be defined also if $(S_\infty,T_\infty) \in \Lambda_0$; in that case $t_E=0$, so that all points in $\tilde \lambda_0$ are fixed points of $\tilde G$.

In the original time scale the return time between one epidemic and the next one is approximately equal, as $\delta -> 0$, to $t_E/\delta$, since the time of an epidemic is negligible in this limit.

On the other hand, if $\beta (S_\infty + \alpha \nu (P_\infty+ T_\infty)\le \gamma$, the system enters the slow time scale close to
$(S=S_{\rm fin}  = S_\infty, P = P_{\rm fin} = P_\infty + T_\infty, T = 0)$
In this time-scale, the system moves from he point $(S_{\rm fin} , P_{\rm fin}) \in \Lambda_-$ to a point $(S(T_E), P(T_E)) \in \Lambda_+$, where the functions $S(x)$ and $P(x)$ are shown in \eqref{sol_slow} with $P_\infty = P_{\rm fin}$ and $R_\infty = 1 - S_{\rm fin} - P_{\rm fin}$, while $T_E$ is found by solving \eqref{entry_exit}.

We can then define
 $G_1(S_{\rm fin} , P_{\rm fin}) =S(T_E)$, $G_2(S_{\rm fin} , P_{\rm fin})= P(T_E)$ and extend this to a function 
$\tilde G : \tilde \Lambda^- \to \tilde \Lambda^+$ through
\begin{equation}
    \label{G_case2}
   \tilde G_1(S,P,T) = G_1(S,P+T),\qquad \tilde G_2(S,P,T) = G_2(S,P+T),  \qquad \tilde G_3(S,P,T) = 0.
   \end{equation}

In summary, we can summarize the behaviour of the system between the start of an epidemic and the start of the next one through a discrete map $G \circ F$ where
$$ \tilde F : \tilde \Lambda_+ \to \tilde \Lambda_- $$
is defined through \eqref{def_F_tilde}, while
$$ G : \tilde \Lambda_- \to \tilde \Lambda_+ $$
that has two possible definitions (\eqref{G_case1} or \eqref{G_case2}) depending in whether $(S_\infty,P_\infty + T_\infty)$ is in $\Lambda_+$ or not.

In Figure \ref{fig:disc_cont}, we compare the singular solutions built through the discrete maps to the numerical solutions of \eqref{red_sys} computed with $\delta$ and $\varepsilon$ small. It appears that indeed the singular solutions approximate well \eqref{red_sys} for reasonable values of $\delta$ and $\varepsilon$.
\begin{figure}[H]
   \begin{subfigure}[t]{0.49\textwidth}
 \includegraphics[width=\textwidth]{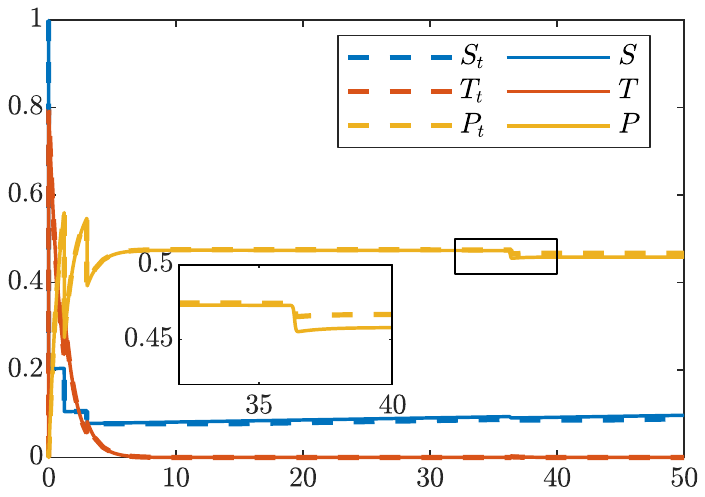}
     \caption{}
     \label{fig:disc1}
\end{subfigure}
\hfill
\begin{subfigure}[t]{0.49\textwidth}
 \includegraphics[width=\textwidth]{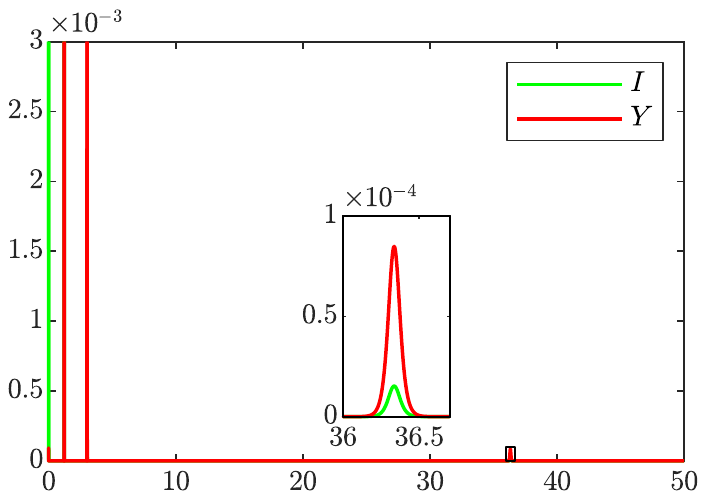}
         \caption{ }
     \label{fig:disc2}
\end{subfigure}
\caption{(a) comparison of the discrete mappings (dashed lines, subscript $t$) with the numerical integration of system \eqref{red_sys} (solid lines). Parameter values are $\beta=2$, $\alpha=0.8$, $\gamma_1 = \gamma_2 = 1$, $\nu = 1.1$, $\delta = 10^{-3}$, $\varepsilon= 4.8\times 10^{-5}$; initial values are $(S,I,T,P,Y) = (0.999,10^{-5},10^{-3},0,10^{-5})$. The units in the $x$-axis correspond to the intermediate time-scale $\tau_1$. (b) The infectives $I$ and $Y$ in the numerical solution of system \eqref{red_sys}. Notice a big epidemic dominated by $I$ at $\tau_1\approx 0$, followed by two large epidemics, dominated by $Y$ at times $\tau_1 < 5$; afterwards, in the discrete approximation, the system enters the slow time scale until a smaller epidemic at $\tau_1 \approx 36$, visible in the inset. This epidemic corresponds to a ``dip'' in the time series of $P$, enlarged in the inset of (a).}
\label{fig:disc_cont}
\end{figure}

\section{Numerical Simulations}\label{sec:numeric}
\subsection{Bifurcation analysis}\label{sec_bifurc}

The bifurcation analysis of system \eqref{red_sys} was carried out with MATCONT \cite{dhooge2003matcont}. We focus on the role of $\beta$, the infection rate of totally susceptible individuals by first time infectious individuals, and its interplay with $\alpha$, the multiplicative parameter which distinguishes infectiousness of secondary vs.\ primary infections. We showcase how $\beta$ influences the stability of the endemic equilibria, and, in particular, the value(s) of $I$ at the equilibria.

From the analysis in Section \ref{sec_eqstab} we know that, if $\alpha \nu > 1$, we can distinguish between three parameter regions: for $\RO < R^*$, the only equilibrium is the DFE (which we proved to be globally stable when $\RO < 1/\alpha\nu$); for $R^* < \RO < 1$, the DFE is still locally asymptotically stable, but there exist also two endemic equilibria; for $\RO > 1$, the DFE is unstable and there exists a unique endemic equilibrium. 

\begin{figure}[H]
     \centering
     \begin{subfigure}[b]{0.49\textwidth}
         \centering
       \begin{tikzpicture}
   \node at (0,0){\includegraphics[width=\textwidth]{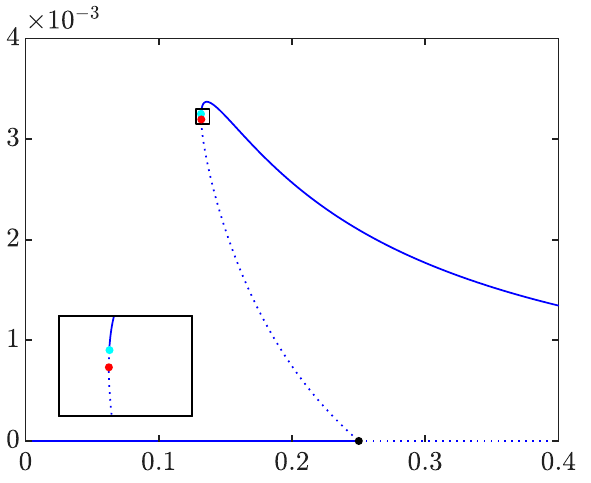}};
    \node at (-0.8,-3.2) {\large $\beta$};
    \node at (-4,-0.5) {\large $I$};
   \node at (1.1,-2.4) {BP};
    \node at (-2.2,-1.4) {H};
    \node at (-2.1,-1.8) {LP};
	\end{tikzpicture}
            \caption{\label{bifurc_a}}
    \end{subfigure}
     \hfill
     \begin{subfigure}[b]{0.49\textwidth}
         \centering
    \begin{tikzpicture}
   \node at (0,0){\includegraphics[width=\textwidth]{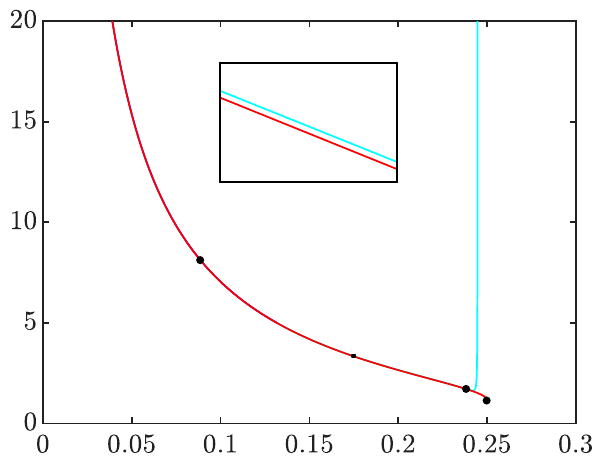}};
      \node at (-0.8,-3.2) {\large $\beta$};
    \node at (-4,-0.5) {\large $\alpha$};
  \node at (-1,0) {BT};
   \node at (2.1,-1.8) {BT};
   \node at (3,-2.2) {C};
	\end{tikzpicture}
             \caption{\label{bifurc_b}}
      \end{subfigure}
     \caption{Values of the parameters: $\beta$ varying as shown, $\alpha = 5$ (a), $\nu = 0.9$, $\gamma_1 = \gamma_2 = 0.25$, $\delta = 0.05$, $\varepsilon = 0.05$. (a): $\beta$ against values of $I$ at equilibrium, both DFE ($I=0$) and EE; solid line: stable; dashed line: unstable. At $\beta=0.25$, $\RO=1$, and the system exhibits a Branching Point (BP). Inset: zoom-in in the small region containing a subcritical Hopf bifurcation (H), from which an unstable branch of limit cycles arises, and a Limit Point (LP); (b) two parameter bifurcation diagram in $\beta$ and $\alpha$, continuing the LP (red) and the H (cyan)  from figure (a).}
     \label{bifurc}
\end{figure}
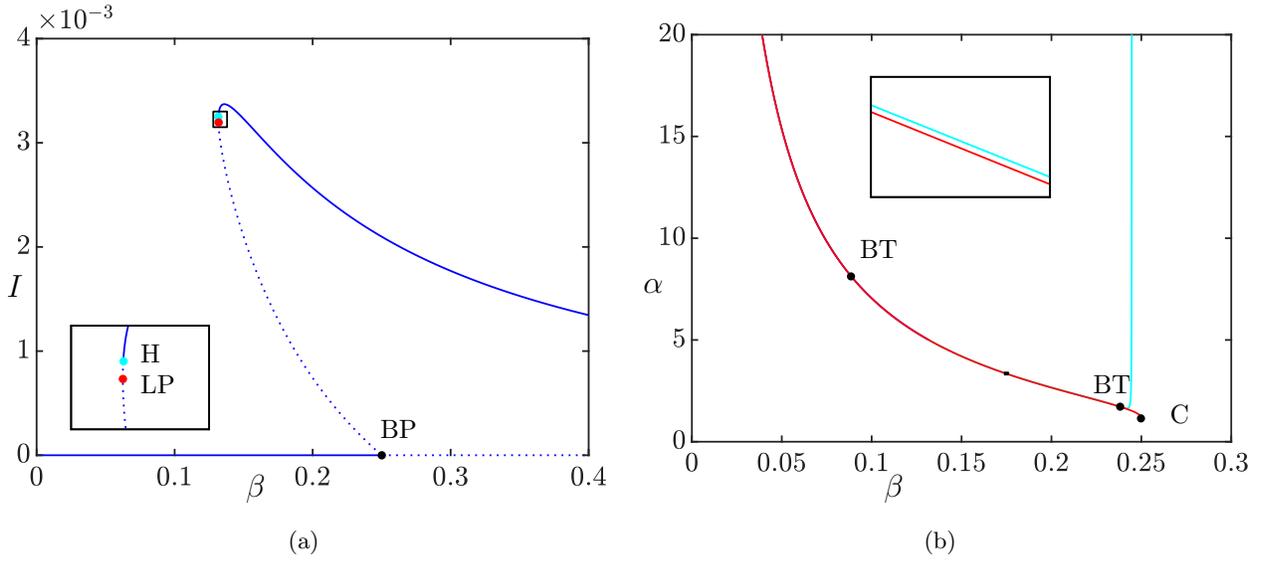

 The bifurcation analysis in Figure \ref{bifurc_a} illustrates these results, and also allows to study the stability of the endemic equilibria. For $\RO > 1$ (corresponding to $\beta > \beta_\text{BP} = 0.25$), the unique endemic equilibrium is always asymptotically stable, while the DFE is unstable; the DFE becomes asymptotically stable as $\beta$ decreases through $\beta_\text{BP}$ with a transcritical (backward) bifurcation, giving rise to a branch of (unstable) endemic equilibria for $\RO < 1$; this branch turns around at $\beta = \beta_\text{LP}$ through a fold (saddle-node) bifurcation. Hence, for $\beta \in (\beta_\text{LP},\beta_\text{BP})$  there are two endemic equilibria as proved in Section \ref{sec:endemic};   the upper endemic equilibrium arises at $\beta_\text{LP}$ as an unstable equilibrium and becomes stable (through a subcritical Hopf bifurcation) at $\beta = \beta_\text{H}$. For $\beta > \beta_\text{H}$, the  upper endemic equilibrium is asymptotically stable. 

From this analysis, we deduce that if $\beta $ belongs to the (very small) interval $(\beta_\text{LP},\beta_\text{H})$, both endemic equilibria are unstable, and presumably all solutions are attracted to the DFE. 

The phenomenon is further investigated in Figure \ref{bifurc_b}, where we present a two parameter bifurcation diagram in $\beta$ and $\alpha$. Figure \ref{bifurc_b} shows a curve (in red) of fold bifurcation (LP) points and another (in cyan) of Hopf bifurcation points; the two curves intersect at the two Bogdanov-Takens (BT) points (with purely imaginary eigenvalues), and are close but separate otherwise (see inset). The LP curve ends in a Cusp point (C), where it joins the DFE at $\RO = 1$ and $\alpha \nu = 1$; the H curve is continued by MATCONT beyond both BT; however, in these regions the H curves actually represent Neutral Saddles, and not Hopf points, which only exist between the two BT. There could exist  a curve of homoclinic bifurcation points between the two BT points, where the unstable periodic solutions arising from the Hopf points disappear; however, we were not able to compute this through the use of MATCONT.

Fixing $\alpha=5$, from Figure \ref{bifurc_a} we select two values of $\beta$, namely $\beta=0.1324$ and $\beta =0.15$. We refer to Figures \ref{fig:both_spiral} and \ref{fig:both_other}, respectively, for a visualization of projections on the $(S,I)$ plane of orbits of the system with these values of the parameters. 

In Figure \ref{fig:both_spiral}, we observe that the endemic equilibrium is \emph{almost} globally asymptotically stable. Even though $\RO=0.6<1$, the basin of attraction of the DFE is rather small, compared to the one of the EE.

In Figure \ref{fig:both_other}, we observe a non-trivial distribution of points belonging to the basins of attraction of the EE and of the DFE. This ``mixing'' is due to the fact that we are observing a 2D projection of 5D orbits.

\subsection{Endemic equilibrium with $\RO<1$}\label{sec_num}

In Theorem \ref{existence_EE}, we proved the existence of the endemic equilibrium even when $\RO < 1$. However, as a consequence of Lemma \ref{lemma_DFE} and Theorem \ref{glob_stab}, we know that there exists a region of the parameter space in which $\RO <1$ but the system may converge to the endemic equilibrium.

In this section, we perform some numerical simulations of the model in order to illustrate this behaviour. We start with $50$ random initial conditions and we plot the trajectories of the system in the plane $(S,I)$. {This means that we are projecting the full 5D system \eqref{red_sys} onto a 2D manifold, which explains the seemingly overlapping orbits.} Varying the value of $\RO$ and the product $\alpha \nu$, we obtain different scenarios.

We use the same parameters fixed in the bifurcation analysis and we set $\alpha=5$. The value of $\beta$ changes accordingly to the bifurcation diagram presented in Figure \ref{bifurc_a}.

In Figures \ref{fig:both_spiral} and \ref{fig:both_other}, the system converges to both equilibria, depending on the initial conditions. {The behaviour in the two cases is different: in \ref{fig:both_other} there is a clear separation of the trajectories which converge to the EE and to the DFE, while in \ref{fig:both_spiral} there is no separation, which is due to the non-trivial contribution of the variables which we are not plotting.}

\begin{figure}[H]
   \begin{subfigure}[b]{0.49\textwidth}
 \includegraphics[width=\textwidth]{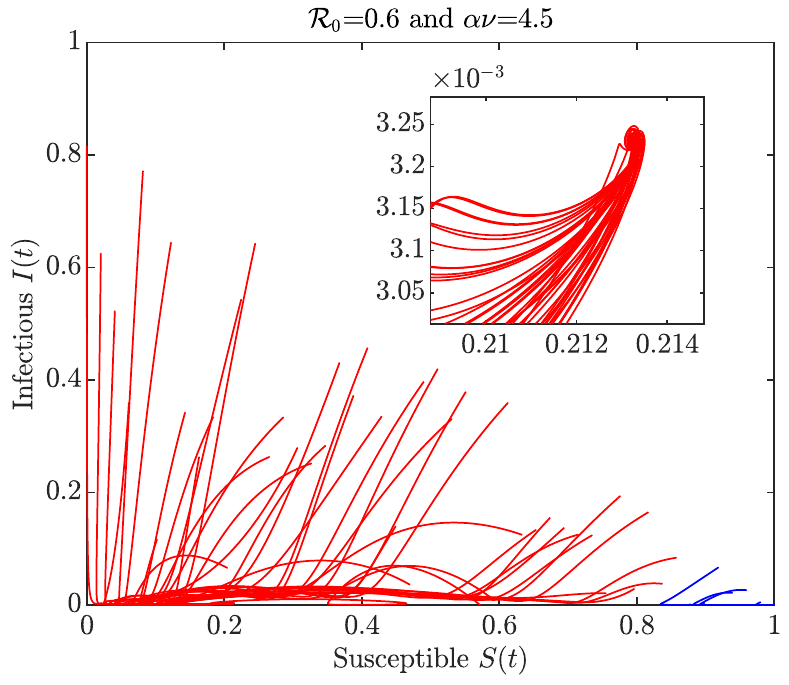}
     \caption{$\beta=0.1322$}
     \label{fig:both_spiral}
\end{subfigure}
\hfill
\begin{subfigure}[b]{0.49\textwidth}
 \includegraphics[width=\textwidth]{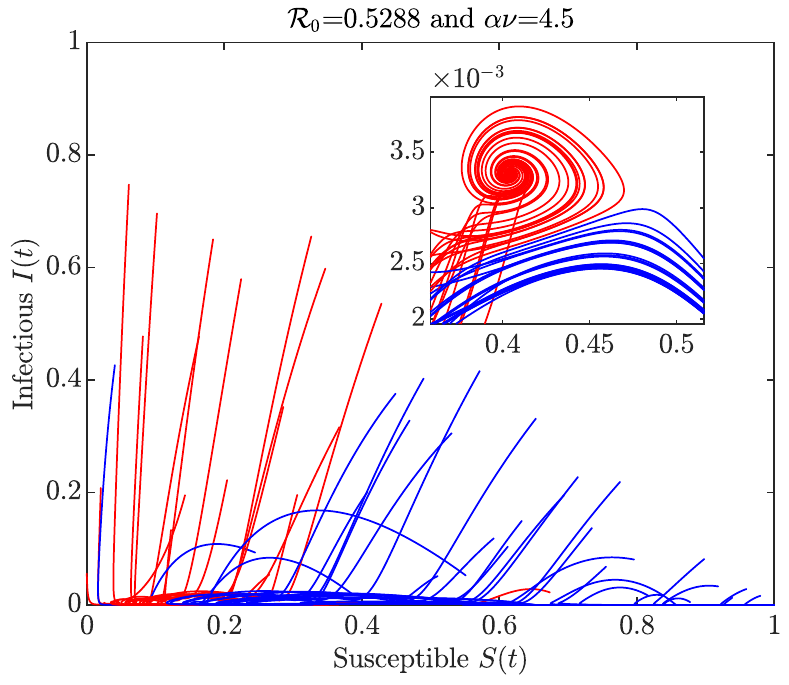}
         \caption{$\beta=0.15$}
     \label{fig:both_other}
\end{subfigure}
\caption{Trajectories on the plane $(S,I)$ starting from $50$ different initial conditions. Blue trajectories converge to the Disease Free Equilibrium, red orbits to the Endemic Equilibrium.}
\end{figure}

\subsection{The role of partial immunity}
In this Section, we analyse the role of partial immunity by varying the value of the parameter $\nu$; for each value of $\nu$, we compare the numerical integration of system \eqref{red_sys} with the discrete mappings described in Section \ref{sec_map}. 

All simulations start with the introduction of the infection in an almost completely susceptible population; hence there is immediately a very big epidemic, represented through the (almost) vertical lines at the left of each figure, in which $S$ decreases (in the fast time scale) from 1 to around 0.2 (as $\RO = 2$). We then show the the three variables, $S$, $P$ and $T$, evolving in the intermediate time scale $\tau_1$. If $\nu = 0$, the second epidemic occurs at $\tau_1$ beyond 200 and is very large, as can be seen by the values (around 40\%) reached by $T$; then a third large epidemic occurs after another long interval, and the system converges very slowly (not shown) towards the endemic equilibrium. Increasing the value of $\nu$, the second epidemic occurs sooner and is smaller, and convergence to the endemic equilibrium, via damped oscillation, occurs much faster. Figures \ref{fig:nu0}, \ref{fig:nu01}, \ref{fig:nu02} and \ref{fig:nu03} illustrate the cases of $\nu=0$, $0.1$, $0.2$ and $0.3$, respectively.

\begin{figure}[H]
   \begin{subfigure}[t]{0.49\textwidth}
     \caption{$\nu = 0$}
 \includegraphics[width=\textwidth]{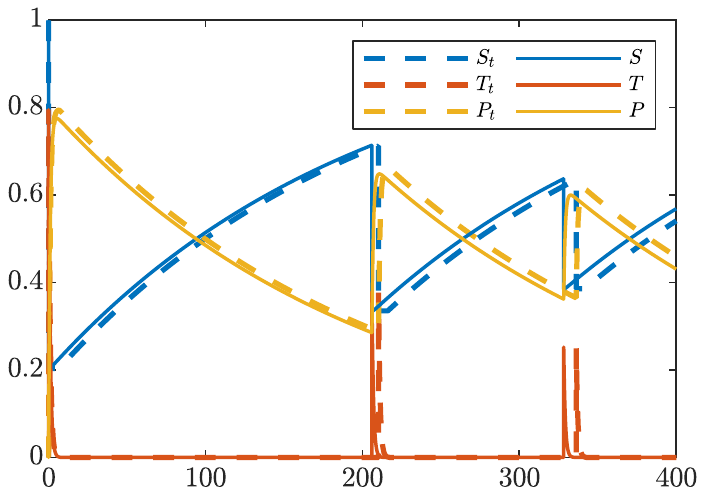}
     \label{fig:nu0}
\end{subfigure}
\hfill
\begin{subfigure}[t]{0.49\textwidth}
         \caption{$\nu = 0.1$}
 \includegraphics[width=\textwidth]{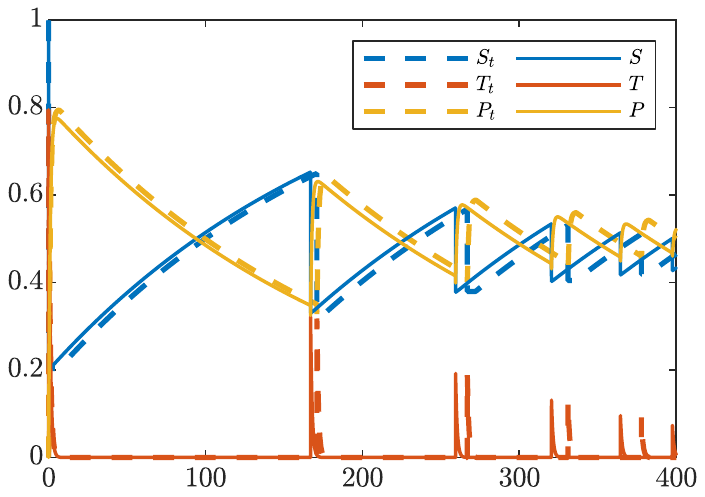}
     \label{fig:nu01}
\end{subfigure}
\hfill
\begin{subfigure}[t]{0.49\textwidth}
       \caption{$\nu = 0.2$}
   \includegraphics[width=\textwidth]{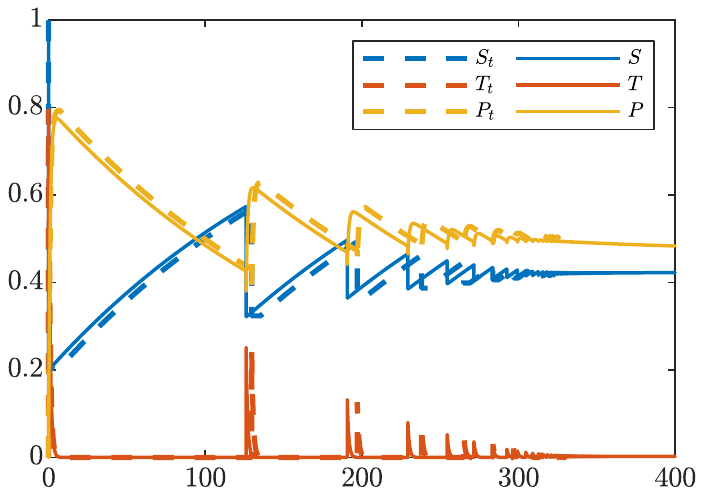}
     \label{fig:nu02}
\end{subfigure}
\hfill
\begin{subfigure}[t]{0.49\textwidth}
         \caption{$\nu = 0.3$}
 \includegraphics[width=\textwidth]{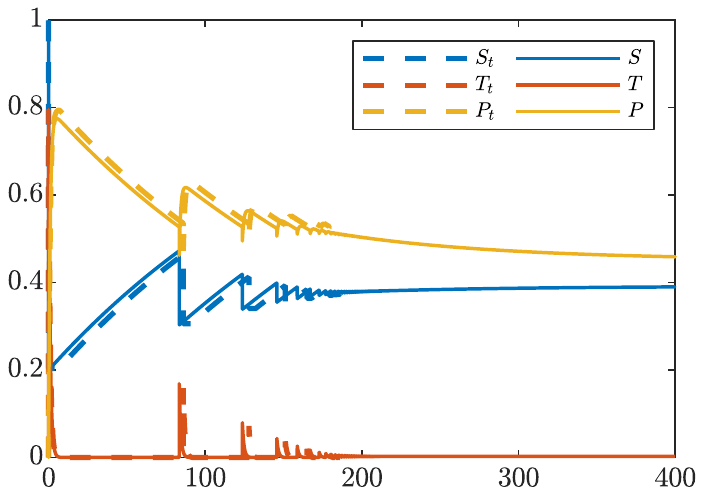}
     \label{fig:nu03}
\end{subfigure}
\caption{Comparison of the discrete mappings (dashed lines, subscript t) with the numerical integration of
system \eqref{red_sys} (solid lines) at varying of $\nu$. Parameter values are $\beta=2$, $\alpha=0.8$, $\gamma_1 = \gamma_2 = 1$, $\delta = 5 \times 10^{-3}$, $\varepsilon= 5 \times 10^{-5}$; initial values are $(S,I,T,P,Y) = (0.999,10^{-5},0,0,10^{-5})$. }
\label{fig:disc_nu}
\end{figure}

It has to be noted that, when the solution is close to the endemic equilibrium, the discrete approximation breaks down, as the solutions no longer arrive at $O(\varepsilon)$- distance from the critical manifold. This, as well as the fact that $\delta$ is small but not infinitesimal, can explain some minor disagreements between the system and their discrete approximations.
\section{Conclusions}\label{sec_conc}
In this paper, we proposed and analysed a model which describes the evolution of a disease with secondary infections. From our assumptions on the parameters, this system naturally involves three distinct time scales.
The interplay between multiple time scales creates previously undocumented phenomena, such as the occurrence of epidemics at different distances in time; for instance, looking at Figure \ref{fig:disc_cont}, one sees that a large epidemic, occurring with the introduction of infection in a totally susceptible population, is followed, after a short interval, by a second and a third epidemic wave; then, there is a very long latent period before the next wave.

The basic reproduction ratio $\RO$ depends only on the parameters relative to the primary infection, $\beta$ and $\gamma_1$. 
However, the parameters, $\alpha$, $\nu$ and $\gamma_2$, involved in a secondary infection, contribute to what we called the ``fast'' reproduction ratio, $\RO^f$, which determines the possibility of an epidemic with a certain faction of totally and partially susceptible individuals.

Moreover, the parameters $\alpha$, $\nu$ and $\gamma_2$, involved in a secondary infection, may allow for a sub-threshold endemic equilibrium; indeed, when $\alpha\nu\frac{\gamma_1}{\gamma_2} > 1$ (and $\delta,\varepsilon \approx 0$), the bifurcation of the DFE at $\RO = 1$ is backwards, giving rise to a branch of positive equilibria for $\RO < 1$, similarly to what obtained in \cite{steindorf2022modeling}.

From a biological point of view, a secondary infection should be milder, so that one expects $\alpha \le 1,\ \nu \le 1,\ \gamma_2 \ge \gamma_1 $; under such assumptions, backward bifurcation cannot occur. However, exactly because a secondary infection is milder, it is possible that individuals have more contacts during a secondary than in a primary, and isolate themselves for shorter periods, thus leading to $\alpha > 1$ and $\gamma_1 > \gamma_2$. Disease-induced mortality (neglected in the model for the sake of simplicity), much higher in a primary than in a secondary infections, would also lead to $\gamma_1$ becoming larger. Therefore, it seems reasonable to assume that $\alpha\nu\frac{\gamma_1}{\gamma_2} > 1$ in certain cases, leading to bistability in the system for $R^* < \RO < 1$ .

Through Figure \ref{fig:disc_nu} we study the effect of partial susceptibility; we show that introducing even a limited susceptibility of individuals recovered from a primary infection has a strong stabilizing effect on infection dynamics. 
When $\nu = 0$, the system goes through a long period with extremely low infection prevalence interspersed with a few large epidemics, before eventually settling to the endemic equilibrium; if $\nu$ is increased to $0.1$-$0.3$, the convergence to the endemic equilibrium is much faster and the epidemic waves, following the first one, are much less intense.

The stabilizing effect of partial immunity can be seen also by comparing the results shown in the bifurcation analysis, Fig.~\ref{bifurc}, with what had been found in the SIRWS model \cite{Dafilis2012a}. In the parameter region that we explored (that includes cases with $\alpha$ and $\nu$ very different from those of Fig.~\ref{bifurc}), the unique (resp. upper) endemic equilibrium for $\RO > 1$ (resp. $R^* < \RO < 1$)  is asympotically stable, except for a tiny interval when $ \RO \approx R^*$. On the other hand, Dafilis \textit{et al.} \cite{Dafilis2012a} found supercritical Hopf bifurcation points for a large interval of $\nu$ values. As mentioned in the Introduction, setting $\alpha =0$ the current model becomes very similar to an SIRWS model, except for the fact that an infection provides only partial immunity, while complete immunity is reached only after a boosting episode. Hence, we believe that partial immunity after a primary infection is the main reason why the results obtained on the stability of the endemic equilibrium differs from those in \cite{Dafilis2012a}.

From a mathematical point of view, our analysis relied mostly on geometric singular perturbation theory. Moreover, we made extensive use of the so-called \emph{entry-exit function}, in a novel setting involving three time scales, in order to distinguish between the cases where the slowest time scale manifests itself in the limiting behaviour of the system or not, using geometric criteria.

A natural yet burdensome generalization of the model we analysed here could include e.g. additional mortality rate in both infectious compartments. However, this would significantly increase the complexity of the model, unless a system of ODEs for the fraction of individuals in each compartment is developed first, as a dimensionality reduction would be more challenging. On the other hand, this would shift the challenges from an additional dimension of the system to a more complex set of ODEs.

Alternatively, one could generalize our modelling approach to a compartmental model describing $n$ consecutive infections. This has been done already for systems with no explicit time scale separation but not, to the best of the authors' knowledge, for system evolving on multiple time scales. In \cite{guo2012global}, for example, the authors assume that infectious individual can move both forward and backward on the chain of stages, in order to incorporate both a natural disease progression and the amelioration due to the effects of treatments. On the other hand, in \cite{bichara2014global}, the authors consider an $n$ strain model, both without immunity and with immunity for all the strains. \\
\\
\noindent \textbf{Acknowledgements.} Panagiotis Kaklamanos was supported by the EPSRC grant
``EP/W522648/1 Maths Research Associates 2021''. Andrea Pugliese, Mattia Sensi and Sara Sottile were supported by the Italian Ministry for University and Research (MUR) through the PRIN 2020 project ``Integrated Mathematical Approaches to Socio-Epidemiological Dynamics'' (No. 2020JLWP23).

Andrea Pugliese and Sara Sottile are members of the ``Gruppo Nazionale per l'Analisi Matematica e le sue Applicazioni" (GNAMPA) of the ``Istituto Nazionale di Alta Matematica" (INdAM).

\appendix
\section{Proof of Theorem \ref{existence_EE}}\label{app:proof}
In order to compute the expression of the endemic equilibrium, we consider system \eqref{sys_simpler}.
We first notice that 
$$
    T'=0 \implies T= \frac{\gamma_1}{\varepsilon} I \; \text{ and } \;  R'=0  \implies R =\frac{\gamma_2}{\delta \varepsilon} Y.
$$
Moreover,
$$
Y' = 0 \implies \nu \beta P (I + \alpha Y ) = \gamma_2 Y. 
$$
We substitute the previous equations in the equation $P'=0$, obtaining
$$
\gamma_1 I - \gamma_2 Y - \delta \varepsilon P + \gamma_2 Y = 0  \implies I = \frac{\delta \varepsilon}{\gamma_1} P \; \text{ and } \; T = \delta P .
$$
From the equations $I'=0$ and $Y'=0$ we know that 
$$I+\alpha Y = \dfrac{\gamma_1 I}{\beta S} \; \text{ and }\; I+\alpha Y = \dfrac{\gamma_2 Y}{\nu \beta P},$$
thus
$\nu \frac{P}{S} = \frac{\gamma_2 Y}{\gamma_1 I}$, from which 
\begin{equation}\label{ypsilon_end}
Y = \frac{\nu \gamma_1}{\gamma_2} \frac{P I}{S} = \frac{\delta \varepsilon \nu}{\gamma_2 }\frac{P^2}{S}.
\end{equation}

Then, 
\begin{equation}\label{S_start}
I'=0  \implies S = \frac{\gamma_1}{\beta} \left( 1- \frac{\beta \alpha \nu P}{\gamma_2}\right).
\end{equation}

Substitute in \eqref{ypsilon_end}, we obtain
$$
Y = \frac{\delta \varepsilon \nu}{\gamma_2} \dfrac{P^2}{\dfrac{\gamma_1}{\beta} - \dfrac{\gamma_1 \alpha \nu}{\gamma_2}P} \; \text{ and } \; R= \dfrac{\nu P^2}{\dfrac{\gamma_1}{\beta} - \dfrac{\gamma_1 \alpha \nu}{\gamma_2}P}.
$$
Notice that, since by assumption on the endemic equilibrium $Y,R > 0 $, we need 
\begin{equation}\label{P_condi}
P < \frac{\gamma_2}{\beta \alpha \nu} = \frac{Q}{ \RO},
\end{equation}
where 
$$ Q = \frac{\gamma_2}{\gamma_1 \alpha \nu  }$$
is a quantity that will be used repeatedly, and  $\RO = \beta/\gamma_1$.

Since we know that $ S+ I + T + P + Y + R = 1$, we can substitute and obtain
\[
\left(1-\frac{1}{Q }+\dfrac{\delta \varepsilon}{\gamma_1} + \delta \right) P + \left(\dfrac{\delta \varepsilon}{\gamma_2}+1\right) \left(\dfrac{\nu P^2}{\dfrac{1}{\RO} - \dfrac{P}{Q}} \right) = 1-\dfrac{1}{\RO}.
\]
For ease of notation, we rename $x := P$. Then, finding an endemic equilibrium is equivalent to finding the solution of 
\begin{equation}
A x + B  \dfrac{x^2}{\dfrac{1}{Q}- x} - 1 +\dfrac{1}{\RO} = 0,
\label{f_endemic}   
\end{equation}
where $A=\left(1-\dfrac{1}{Q} +\dfrac{\delta \varepsilon}{\gamma_1} + \delta \right) $ and $B= \dfrac{\gamma_2 + \delta \varepsilon}{\gamma_1 \alpha} >0$.

After multiplying each side by $\left(\dfrac{1}{Q}- x\right)$ and rearranging the terms, we can rewrite \eqref{f_endemic} as $f(x)= 0$, where
\begin{equation}
f(x) := (B-A)x^2 + \left(1- \dfrac{1}{\RO} + A\dfrac{Q}{\RO} \right)x +\dfrac{Q}{\RO} \left(\dfrac{1}{\RO}-1\right) = ax^2 + b x + c = 0.
\label{def_f_endemic}   
\end{equation}
where the coefficients of the second degree polynomial are:
\begin{equation}
 \begin{split}
\label{abc}
a &= B- A =  \nu Q - 1 + \frac1Q +O(\delta),  \\
b &= 1 - \frac{1}{\RO } + A \frac{Q}{\RO }=1 - \frac{2}{\RO} + \frac{Q}{\RO}+O(\delta),\\
c &= \frac{Q}{\RO} \left(\frac{1}{\RO} -1 \right).
\end{split}   
\end{equation}

We identify the following cases:
\begin{itemize}
    \item Case 1:   $\RO>1$.

If $\RO >1$, then $f(0)=c <0$; we can rewrite
$$ f(x) = B x^2 + Ax \left(\dfrac{Q}{\RO} - x\right)+ \left(1- \dfrac{1}{\RO} \right) \left(x- \dfrac{Q}{\RO}\right),$$
which implies
$$ f\left(\frac{Q}{ \RO}\right) = B \left(\frac{Q}{ \RO}\right)^2 > 0. $$

Thus, there exists a unique positive solution of $f(x)=0$ in the interval $\left(0,\frac{Q}{ \RO}\right)$, and we can conclude 
\begin{lemma}
\label{lemma_R0.gt.1}
If $\RO >1$, the system admits a unique EE.
\end{lemma}
\item Case 2: $\RO<1$. 

First of all, we rewrite $f$ as a function also of $\RO$ as
$$ G(x,\RO) = (B-A)x^2 + \left(1- \dfrac{1}{\RO} + A\dfrac{Q}{\RO} \right)x +\dfrac{Q}{\RO} \left(\dfrac{1}{\RO}-1\right). $$
It is immediate to see that $G(0,1) = 0$. As long as
$G_x(0,1) = A Q \not = 0$, the implicit function theorem let us find a branch of equilibria $x(\RO)$ such that $x(1) = 0$. Precisely, we obtain
$$x'(1)= - \frac{G_{\RO}(0,1)}{G_x(0,1) } = \frac{Q}{ AQ}=\frac{1}{A } = \frac{1}{\left(1-\dfrac{1}{Q} +\dfrac{\delta \varepsilon}{\gamma_1} + \delta \right) }.$$
If $x'(1) > 0$, it is impossible that $x'(R) < 0$ for some $R > 1$. Otherwise, one would have two solutions of \eqref{def_f_endemic} for some $\RO > 1$, and this goes against 
Lemma \ref{lemma_R0.gt.1}.

 In the case $G_x(0,1)=0$, one can obtain, looking at the second derivative of $\RO(x)$, the same result. 
 
 Hence, we obtain the following Lemma:
\begin{lemma}
\label{lem_backward_bif}
There exist solutions of $f(x)=0$ with $\RO < 1$ if and only if $G_x(0,1) = AQ < 0$.\\
Since $Q > 0$ and $A=\left(1-\dfrac{1}{Q} +\dfrac{\delta \varepsilon}{\gamma_1} + \delta \right) = \left(1-\dfrac{1}{Q} \right)+ O(\delta)$, this occurs if and only if $Q < 1$ when $\delta \approx 0$.
\end{lemma}
We now find which are exactly the values of $\RO$ for which such solution exist.

Since $\RO < 1$, we have
$$ f(0) >0,\qquad f\left(\frac{Q}{\RO}\right) > 0, $$
and we ask whether there exist two roots of $f(x) = 0$ in $\left(0,\dfrac{Q}{\RO}\right)$.

This occurs if and only if
$$ a > 0,\; b < 0,\; b^2 > 4ac,\; -\frac{b}{2a} < \frac{Q}{\RO}.$$

Through some computations, we see that, using the definitions in  \eqref{abc}, the conditions are equivalent, for $\delta \approx 0$, to
\begin{equation}
\label{conds}
\left\{\begin{aligned}
\nu Q - 1 + \frac1Q &> 0,  \\
\RO &< 2 - Q , \\
\RO &> Q - 2 \nu Q^2,  \\
(\RO - Q)^2 &> 4 \nu Q^2 (1-\RO).
\end{aligned}
\right.\end{equation}
As a consequence of Lemma \ref{lem_backward_bif}, we can assume $Q < 1$, so that the first two conditions are immediately satisfied.

Finally, the conditions can be rewritten as
\begin{equation}
\label{condfin}
\left\{\begin{aligned}
Q &< 1, \\
\RO &> Q (1- 2 \nu Q), \\
(\RO - Q)^2 &> 4 \nu Q^2 (1-\RO).
\end{aligned}
\right.\end{equation}

If we compute explicitly the third condition, we obtain:
$$
\RO^2 + 2Q(2\nu Q - 1) \RO +Q^2(1-4\nu) > 0,
$$
whose discriminant $\Delta$ is computed as
$$\begin{aligned}
    \dfrac{\Delta}{4} = Q^2 (2\nu Q - 1)^2 - Q^2 (1 - 4\nu)
                    = Q^2 \left(4 \nu^2 Q^2 - 4 \nu Q + 4 \nu \right)
                    =  4\nu Q^3 \left( \nu Q - 1 + \dfrac{1}{Q}\right). 
\end{aligned}$$

We notice that $\Delta > 0$ due to the first condition of \eqref{condfin}. Thus we obtain
$$
\RO < Q (1- 2 \nu Q) - \sqrt{\frac{\Delta}{4}}\quad \text{ or }\quad \RO >Q (1- 2 \nu Q)+\sqrt{\frac{\Delta}{4}}.
$$
If $\RO$ satisfies  the first inequality, it cannot satisfy the second condition of \eqref{condfin}; hence, we   obtain that the second and third conditions of \eqref{condfin} are equivalent to
$$
\RO >  Q (1- 2 \nu Q) +\sqrt{4\nu Q^3 \left( \nu Q - 1 + \dfrac{1}{Q}\right) }.
$$
In summary, we obtain (in the limit of $\delta \approx 0$) that two endemic equilibria exist if and only if
\begin{equation}
\label{cond_endemic}
\left\{\begin{aligned}
    Q <& \; 1, \\
   1> \RO >&  \;Q (1-2\nu Q) +\sqrt{4\nu Q^3 \left( \nu Q - 1 + \dfrac{1}{Q}\right)}.  
\end{aligned}
\right.\end{equation}
We can then rewrite \eqref{cond_endemic} in terms of the parameters, obtaining
\begin{align}
    \alpha \nu >&\; \dfrac{\gamma_2}{\gamma_1 }, \\
  1>  \RO >& \; 2 \dfrac{\gamma_2}{\gamma_1 \alpha \nu} \left(\dfrac{1}{2} - \dfrac{\gamma_2}{\gamma_1 \alpha } + \sqrt{ \nu - \dfrac{\gamma_2}{\gamma_1 \alpha}\left(1- \dfrac{\gamma_2}{\gamma_1 \alpha}\right) } \right).
\end{align}
\end{itemize}
One could repeat the computations for positive values of $\delta$ and $\varepsilon$, but the result would be very cumbersome.

\bibliographystyle{abbrv}
\bibliography{bib}
\end{document}